\documentclass[11pt]{article}
\usepackage[T1]{fontenc}
\usepackage{amsfonts}
\usepackage{amsmath}
\usepackage{amssymb}
\usepackage{amsthm}
\usepackage{bbm}
\usepackage{bm}
\usepackage{mathrsfs}
\usepackage{verbatim}
\usepackage{setspace}
\usepackage{color}
\usepackage{pdfsync}
\usepackage{enumerate}
\usepackage{graphicx}
\usepackage{subfigure}
\usepackage{stmaryrd}
\usepackage{float} %
\usepackage{tikz}
\usetikzlibrary{automata, positioning, patterns}
\usepackage{tcolorbox}
\usepackage[comma,authoryear]{natbib}

\usepackage[margin=1.2in]{geometry}

\theoremstyle{plain}
\newtheorem{theorem}{Theorem}[section]
\newtheorem{proposition}[theorem]{Proposition}
\newtheorem{lemma}[theorem]{Lemma}

\theoremstyle{definition}

\newtheorem{remark}[theorem]{Remark}

\newtheorem{example}[theorem]{Example}

\theoremstyle{remark}

\renewenvironment{thebibliography}[1]{%
\begin{oldthebibliography}{#1}%
\setlength{\baselineskip}{.9em}
\linespread{1}
\small
\setlength{\parskip}{0.3ex}%
\setlength{\itemsep}{.2em}%
}%
{%
\end{oldthebibliography}%
}
\newcommand{\eps}{\varepsilon}

\newcommand{\R}{\mathbb{R}}

\DeclareMathOperator{\avg}{avg}

\DeclareMathOperator{\Var}{Var}

\newcommand{\tsum}{{\textstyle \sum}}

\newcommand{\degsymb}{$^{\circ}\hspace{0pt}$}

\newcommand{\mykill}[1]{}
\newcommand{\MN}[1]{#1}

\numberwithin{equation}{section}

\newlength{\flexcheckerboardsize}

\newcommand{\defineflexcheckerboard}[4]{
    \setlength{\flexcheckerboardsize}{#2}
    \pgfdeclarepatterninherentlycolored{#1}
        {\pgfpointorigin}{\pgfqpoint{2\flexcheckerboardsize}    
        {2\flexcheckerboardsize}}
        {\pgfqpoint{2\flexcheckerboardsize}
        {2\flexcheckerboardsize}}%
        {
            \pgfsetfillcolor{#4}
            \pgfpathrectangle{\pgfpointorigin}{
            \pgfqpoint{2.1\flexcheckerboardsize}    
                {2.1\flexcheckerboardsize}}%
          \pgfusepath{fill}
          \pgfsetfillcolor{#3}
          \pgfpathrectangle{\pgfpointorigin}
            {\pgfqpoint{\flexcheckerboardsize}
            {\flexcheckerboardsize}}
          \pgfpathrectangle{\pgfqpoint{\flexcheckerboardsize}
            {\flexcheckerboardsize}}
            {\pgfqpoint{\flexcheckerboardsize}
            {\flexcheckerboardsize}}
            \pgfusepath{fill}
        }
}

\definecolor{mygreen}{RGB}{0,130,0}
\definecolor{myred}{RGB}{150,0,25}
\definecolor{myorange}{RGB}{255,100,0}

\defineflexcheckerboard{flexcheckerboard_greenred}{2mm}{mygreen}{myred}
\defineflexcheckerboard{flexcheckerboard_greenorange}{2mm}{mygreen}{myorange}
\defineflexcheckerboard{flexcheckerboard_orangered}{2mm}{myorange}{myred}
\defineflexcheckerboard{flexcheckerboard_whitered}{2mm}{white}{myred}
\usepackage[pdfborder={0 0 0}]{hyperref}
\hypersetup{
  urlcolor = black,
  pdfauthor = {Marcel Nutz, Florian Stebegg},
  pdfkeywords = {Climate Change Adaptation under Heterogeneous Beliefs},
  pdftitle = {Climate Change Adaptation under Heterogeneous Beliefs},
  pdfsubject = {Climate Change Adaptation under Heterogeneous Beliefs},
  pdfpagemode = UseNone
}

\begin{document}

\title{\vspace{-2.4em}
  Climate Change Adaptation under Heterogeneous Beliefs
\date{\today}
\author{
  Marcel Nutz%
  \thanks{
  Columbia University, Departments of Statistics and Mathematics, mnutz@columbia.edu. Research supported by an Alfred P.\ Sloan Fellowship and NSF Grants DMS-1812661, DMS-2106056. MN is grateful to Harrison Hong and Jos{\'e} Scheinkman for helpful comments and encouragement.
  }
  \and
  Florian Stebegg%
  \thanks{
  Columbia University, Department of Statistics, florian.stebegg@columbia.edu.
  } 
  }
}
\maketitle \vspace{-1.3em}

\begin{abstract}
We study strategic interactions between firms with heterogeneous beliefs about future climate impacts. To that end, we propose a Cournot-type equilibrium model where firms choose mitigation efforts and production quantities such as to maximize the expected profits under their subjective beliefs. It is shown that optimal mitigation efforts are increased by the presence of uncertainty and act as substitutes; i.e., one firm's lack of mitigation incentivizes others to act more decidedly, and vice versa.
\end{abstract}

\section{Introduction}

There is broad consensus among scientists that anthropogenic emissions of greenhouse gases are the main driving factor for climate change. Nevertheless, there is considerable uncertainty about the magnitude of future climate change and its impacts.\footnote{See \citep{IPCC.18summary, USGCRP.17,  WEF.16}, among others.}
Firms making long-term investments, such as electric utilities planning power plants, face uncertainty about the future regulatory environment. For instance, a utility anticipating carbon taxes may opt for a sustainable technology even if it is more expensive at the time of planning.\footnote{\cite{Economist0727.19} stresses the risk that plants become uneconomic: ``In April, Indiana's utility commission rejected a proposal for a gas plant by Vectren \!\dots for just that reason. If America one day sets a price on carbon emissions, customers could be left paying for utilities' bad bets on fossil fuels.'' BlackRock CEO \cite{Fink.20} warns clients that ``coal is \dots highly exposed to regulation because of its environmental impacts.''} Emissions-related tax rates are endogenous because firms' decisions impact the magnitude of climate change as well as public scrutiny, which in turn influence regulatory decisions.%
\footnote{A similar argument could be made for reputational risks, consumer demand, etc. In the spirit of \cite{BarnettBrockHansen.20}, we use taxes as the single target variable in our highly stylized model.}
The importance of uncertainty in climate change economics has been emphasized in the recent literature; see for instance \citep{BrockHansen.17,GillinghamEtAl.18} for recent surveys with numerous references. There are also various works on game-theoretic aspects of climate change mitigation---mostly focusing on whether, or under which circumstances, sufficient mitigation can be achieved.\footnote{See, for instance, \citep{BarrettDannenberg.12}.}
This literature generally assumes that %
agents are homogeneous; two exceptions are \citep{Brechet.14}\footnote{%
\cite{Brechet.14} consider a variation of Nordhaus' DICE-2007 model with co-existing agents in the framework of model predictive control. The population consists of two agents of equal size, a business-as-usual agent not taking into account their own impacts and an agent solving the full optimization problem. In a numerical simulation, the authors conclude that there exists a strong incentive to play business-as-usual and that total emissions are close to a model with only business-as-usual agents.
}
and \citep{Kiseleva.16}\footnote{%
\cite{Kiseleva.16} formulates a model with evolutionary dynamics and three types of agents, distinguished by whether they believe in anthropogenic climate change (``weak/strong skeptics'') and the possibility of a climate catastrophe (``science-based''). The impact of adaptation and pollution costs is described as well as the evolutionary type dynamics, with a focus on whether climate catastrophe can be prevented in the absence of science-based types, the latter being answered positively. %
}.
Clearly the presence of uncertainty is an important reason for the co-existence of heterogeneous beliefs.\footnote{\cite{PoortingaEtAl.11} conduct an empirical study on public skepticism about climate change in the British population and find that a majority is uncertain what the effects of climate change will be. The authors argue that another reason is that ``climate change is perceptually a distant issue.'' \cite{DiGiuliKostovetsky.14} show that Democratic-leaning firms spend more on corporate social responsibility, including environmental, than Republican-leaning firms.} In this paper we take a first step towards studying strategic interactions between firms that differ in their beliefs about future climate impacts. %
We formulate a tractable Cournot-type equilibrium model where firms make irreversible decisions about production and emissions with the aim of maximizing expected profits. Products are subject to  taxes which are endogenous and uncertain at the time of planning. While the model is very stylized, it highlights a strategic aspect of climate change mitigation and allows us to analyze how firms' beliefs about future taxes influence their own and their competitors' choices, as well as total emissions. One finding is that environmental choices act as substitutes: one firm's lack of action in climate change mitigation incentivizes others to increase their efforts.%

For concreteness, we summarize the state of the climate by the global mean near-surface temperature and a firm's emissions by its carbon dioxide output. As in classical Cournot competitions, one can think of the model as having two periods.\footnote{See \citep{Tirole.88} for background on Cournot models.} In the first, $n$~firms make an irreversible decision about the quantity of good to be produced, say the number and size of power plants. In addition, firms choose a technology: hydroelectrical, coal, etc. We assume that by combining technologies this choice boils down to a continuous parameter $r$ which represents the amount of carbon emitted per unit of good: $r=0$ stands for a zero-emission technology and $r=1$ stands for the ``business-as-usual'' technology with the largest emissions.\footnote{Technology $r=0$ is the ``backstop technology'' in the terminology of \cite{Nordhaus.18}.} The production cost is determined by the technology choice and saving emissions is costly. Firms maximize the profits that will be realized when markets settle in the second period as detailed below. The products per se, say electrical power, are considered undifferentiated and are perfect substitutes from the consumers' point of view. As a result, the price paid by consumers is determined by inverse demand as in standard Cournot models. 
The only reason for differentiation among firms is that firms have heterogeneous beliefs about the tax rate (per unit of carbon) that will prevail in the second period.%
\footnote{See for instance \citep{Chari.18} for arguments that unregulated markets are not able to mitigate climate change. While our description assumed for simplicity that differentiation is exclusively due to taxes, the latter may also stand in for consumers' potential willingness to pay a premium for green technology. Like the taxes, the magnitude of this premium is endogenous and uncertain.}
Taxes are paid by the firms and distributed to consumers lump-sum. We think of the two periods as substantially separated in time, reflecting the long planning horizon e.g.\ for large hydroelectrical or nuclear power plants. Thus, net profits will be modeled as random variables and firms maximize \emph{expected} profits at the time of planning. We treat the determination of prices, taxes and production as occurring at a single point in time, with these quantities representing averages during the planning period.
As pointed out by a referee, this static modeling is inappropriate for a rapid and large climate event; e.g., if consequences of climate change  increase dramatically at a threshold in temperature. Following such an event, the firms' available set of actions and their profit maximization would need to be adapted dynamically.

The specific form of taxes in our model is motivated by two stylized facts that we describe next. Climate science tries to predict the anomaly (temperature change) over a period of time as a function of an emissions scenario (carbon emissions for each year). While it is a high-confidence statement that additional carbon emissions cause temperature to increase, there is significant uncertainty about the precise magnitude: different state-of-the-art models produce significantly different projections even for the same emissions scenario.\footnote{%
See \citep{IPCC.18summary} for a broad survey, Figure 11.25(a) of~\citep{IPCC.13} for temperature predictions made by various climate models for four standardized scenarios, and \citep{PreinEtAl.15} for a survey of climate models. Reasons for the difficulty to forecast temperature include nonlinear dynamics (e.g., laws of convection, saturation of oceans, thawing of ice), size and heterogeneity of the planet, length of the required time horizon, random shocks (e.g., volcano eruptions), and others.}
To obtain a tractable model, we use that the increase $T$ in temperature over a time interval is approximately proportional to the cumulative carbon emissions $K$ over that interval: $T=\alpha K$.\footnote{See~\citep{AllenEtAl.09,MatthewsEtAl.09}. This linearity is the approximate combined result of several nonlinear effects, and valid for regimes of moderate emissions. \cite{MacDougallFriedlingstein.15} explain the phenomenon analytically by the diminishing radiative forcing from CO$_{2}$ per unit mass being compensated for by the diminishing ability of the ocean to take up heat and carbon. 
As pointed out by a referee, these models may  allow for the temperature to revert to a lower level if the world reduced total emissions, as implied by the proportionality of the change of temperature. Because such a scenario seems unlikely and the time horizon would presumably be  greater than the horizon of the model, we do not consider temperature decreases further.} The constant of proportionality $\alpha$ is called transient climate response (TCRE). While the approximate linearity of the carbon--climate response function is robust across a range of models, the value of $\alpha$ is subject to model uncertainty; indeed, Figure~\ref{fi:TCREhistogram} suggests that a broad range of values are reasonable parameters for the models tested.
\begin{figure}[tbh]
\begin{centering}
\includegraphics[width=0.50 \textwidth, trim={11cm 20.3cm 2.1cm 2.5cm},clip]{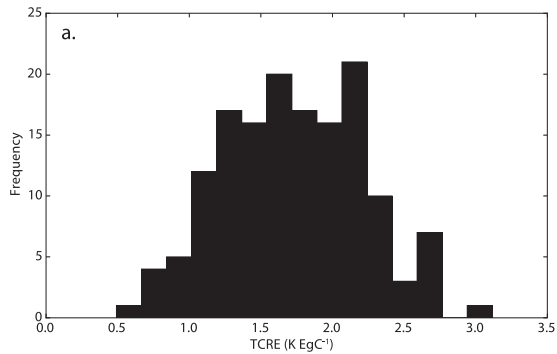}
\caption{Histogram of numerical values of TCRE found from 150 model variants, in \degsymb C  per trillion tonnes of carbon. The 5th
to 95th percentile range is $[0.9,2.5]$. Figure~3(a) of~\citep{MacDougallSwartKnutti.17}, reproduced with permission.}
\label{fi:TCREhistogram}
\end{centering}
\end{figure}
This motivates that we incorporate heterogeneous beliefs about climate change in our model through $\alpha$: firms agree-to-disagree about the TCRE. 
Climate change ``believers'' assume that the TCRE has a relatively higher value whereas climate change ``skeptics'' assume that the value is lower, with a value of zero representing the view that carbon has no impact on temperature. More precisely, firms may acknowledge uncertainty about the correct value of the TCRE and use a probability distribution for $\alpha$. Level and uncertainty are then represented by the mean and the variance of that subjective distribution.\footnote{In particular, the main influencing factor for firms' decisions is the subjective belief about what will happen when markets settle and only incorporates temperature change through carbon emissions. This is highly stylized but allows us to mathematically solve the model. As pointed out by a referee, more realistic models of climate change would also include random shocks. We omit this feature as adding a second family of distributions would complicate the analysis  and we would expect the economic conclusions to be largely equivalent.}

\begin{figure}[tbh]
\begin{center}
\includegraphics[width=.9\textwidth, trim={2.5cm 9.3cm 2.7cm 13.9cm},clip]{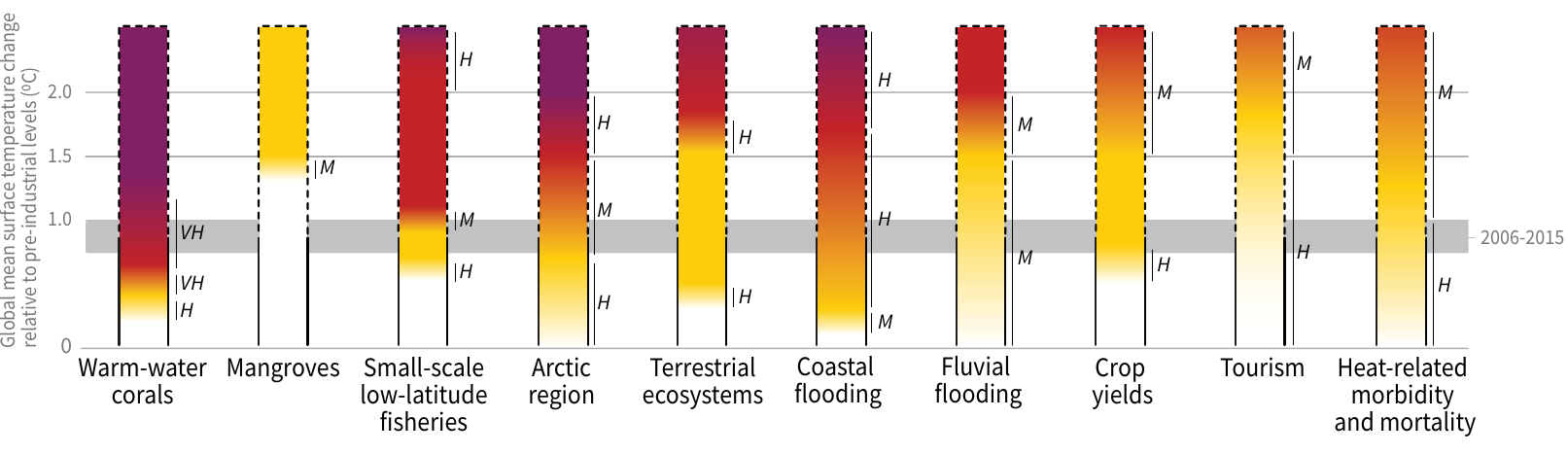}
\caption{Impacts and risks for selected natural, managed and human systems, ranging from ``undetectable'' to ``very high'' (coloring) and confidence levels for the indicated transitions (Low, Medium, High, Very High). Figure SPM.2 of \citep{IPCC.18summary}.}
\label{fi:impacts}
\end{center}
\end{figure}

Next, we describe how warming affects profits. Impacts of global warming are manifold
and admittedly difficult to quantify. Figure~\ref{fi:impacts} illustrates past and expected future impacts on a variety of systems. While the global warming of approximately 1\degsymb C between 1750 and the present has only had limited effect on most of these systems, the next 0.5\degsymb C are expected to have a much more significant impact, and the following 0.5\degsymb C even more so---the marginal impact is increasing.\footnote{See \citep{IPCC.18summary} for a survey on climate impacts. Coral reefs, for example, are expected to decline by >99\% at 2\degsymb C warming. 
The nonlinearity of climate impacts is reinforced by tipping events; i.e., relatively abrupt macroscopic changes in the climate system that are expected at increased temperatures, such as melting of the West Antarctic ice sheet or methane release from thawing permafrost. See \citep{LentonCiscar.13} for a list of nine different tipping events.
\MN{As pointed out by a referee, our focus is on temperature increases even though a large negative change in temperature would also have substantial negative impacts. The general idea that pre-industrial temperature is ``optimal'' is likely less a consequence of the specific temperature than the fact that nature and humans have adapted to it over a long period of time.}
}
This is consistent with the quantitative estimates of climate impact on welfare in the literature. Figure~1 in \citep{Tol.18} compares 27 such estimates and suggests that welfare-equivalent income is a concave function of temperature.\footnote{\cite{Tol.18} emphasizes the uncertainty in these estimates while highlighting that ``impacts of climate change are typically found to be more than linear'' and that the  uncertainty is skewed towards negative surprises.}

Our model postulates an emissions tax introduced by a regulator. This stylized mechanism more generally stands in for potential adverse impacts on firms with larger emissions, such as costs due to additional environmental rules or reputational and legal risks that affect firms in the long run.\footnote{\MN{We phrase the cost of carbon as a tax since this leads to straightforward mechanics in the model. A carbon certificate market where the supply of certificates is controlled by the regulator would lead to similar results: if the firms start with no certificates and need to purchase certificates proportional to the carbon output they} \MN{committed to in the first period, they are price takers in the certificate market. The clearing price is an inverse function of certificate supply and can thus be set indirectly by the regulator.}}
 Specifically, the tax per unit of good produced with technology $r$ is $bT\alpha r$; here $\alpha r$ is the marginal increase in temperature caused by the good (after choosing units appropriately) and the tax rate $bT$ is proportional to the total temperature anomaly.\footnote{\MN{The linear schedule in our setup does not necessarily represent an ``optimal'' choice from the regulator's point of view, but rather a concrete formulation for an expectation of the firms' disutility that makes our model mathematically analyzable. A real regulator would also face a complex decision problem under uncertainty with various trade-offs. In particular, the regulator will want to reduce carbon emissions effectively but also retain flexibility to react to unforeseen developments, as emphasized in \citep{jakob2014optimal}.}}

 The taxes on the total production then internalize  society's disutility from warming if we postulate that the latter is quadratic and separable---the simplest form consistent with increasing marginal disutility.\footnote{The constant $b$ could, in principle, be calibrated to the social cost of warming. However, our stylized model is built to highlight the mechanics of heterogeneous beliefs and not expected to yield good quantitative predictions.}
The tax rate is \emph{endogenous} as it is proportional to the temperature increase~$T$, and uncertainty about the TCRE $\alpha$ implies uncertainty about taxes. Alternately, as warming affects firms only via taxes, we may think of firms as having uncertainty about future regulation per se rather than climate. %
Our modeling of uncertainty and impacts is inspired by~\citep{BarnettBrockHansen.20}, but is even more simplified in order to yield a tractable framework for an equilibrium model.

While closed-form solutions for the equilibrium are available in special cases (see Section~\ref{se:examples}), this model is sufficiently tractable to allow for a detailed analysis of the equilibrium properties in all cases. %
We are particularly interested in comparative statics for firms' choices and aggregate emissions.
As production quantity and carbon interact with the constraints of the model, we find it useful to form buckets of firms with comparable equilibrium technology choices: green firms emit no carbon ($r=0$), red firms make no effort to reduce emissions ($r=1$), and orange firms are intermediate. This color is endogenous, but conditionally on the color, closed-form feedback expressions for the optimal choices are readily available. The optimal production quantity and emissions are substitutes; that is, any firms' optimal choices are decreasing in the total quantity and carbon produced by the other firms. However, if the  total quantity and carbon are varied in opposite directions, the reaction of the firm is ambiguous, with the direction depending on its color. Remarkably, the equilibrium is nevertheless unique; this is shown by an analysis of the interactions between the buckets. Key features of the equilibrium are:

\begin{enumerate}
\item[1.] \emph{Uncertainty is equivalent to higher expected impact.} As far as equilibrium outcomes are concerned, the second moment $\alpha_{i}^{2}:=E_{i}[\alpha^{2}]$ is a sufficient statistic for firm $i$'s belief about climate impacts (or about taxes). A firm acknowledging variance of carbon impact takes the same optimal decisions as one that assumes a known but increased impact.\footnote{This behavior is consistent with the finding that the presence of uncertainty warrants a higher level of climate change mitigation in various contexts; see, e.g., \citep{BergerEmmerlingTavoni.17, BrockHansen.17, Nordhaus.18}.}\footnote{\MN{We could replace  the fixed parameter $b$ with a firm-dependent risk-aversion parameter $b_i$ in our model. This would be redundant in terms of outcomes as $\alpha_i^2$ is already firm-dependent and only enters the equilibrium in the form $\beta_i=b\alpha_i^2$ (which would be modified to $\beta_i=b_i\alpha_i^2$). We therefore choose to use the fixed parameter $b$ to make the interpretation of taxes as a tool to internalize society's disutility more straightforward.}}

\item[2.]  \emph{Higher expected impact implies higher mitigation effort.} For a given firm $i$, the equilibrium technology choice $r_{i}$ is monotone decreasing in $\alpha_{i}^{2}$, if other firms' beliefs are fixed.

\item[3.]  \emph{Mitigation efforts act as strategic substitutes.} For a given firm $i$ with fixed $\alpha_{i}^{2}$, the equilibrium technology choice $r_{i}$ is monotone increasing in $\alpha_{j}^{2}$ for all $j\neq i$. That is, firm~$i$'s mitigation efforts increase if other firms' efforts decrease, and vice versa.
\end{enumerate}

The aggregate carbon emissions are decreasing with respect to the climate beliefs $\alpha_{i}^{2}$ of all firms, as one would expect, but other comparative statics reveal the richer interactions between emissions, production and constraints, both for individual firms and in the aggregate. Strategic substitutes are consistent with the mechanics of a standard Cournot model without technology choice, where one firm's decrease of emissions and production would entail an increase of emissions and production for the other firms. In the present model, technology choice gives an alternate way to control emissions. For instance, Example~\ref{ex:allInt} shows that if firms all have moderate beliefs, firms adapt to changes in beliefs by changing technology while holding production quantities constant. In general, the equilibrium production quantity of a given company depends ambiguously on the beliefs in the economy: If a red firm $j$ grows more concerned about climate impacts, $j$ reduces its emissions and hence, a fortiori, its production quantity. Other firms then face less competition and increase their quantity. If $j$ is orange, however, it reduces emissions by changing technology rather than quantity. Red firms then increase their production since the marginal cost of carbon has decreased, but green and orange firms decrease their production as a reaction to the competition from red firms. A detailed discussion of all comparative statics can be found in the main text.

We also study an iterated version of the game where the total carbon in the environment accumulates and firms can update their beliefs about the TCRE. In an example where firms asymptotically learn (and agree on) the true value of the TCRE, the total temperature change  converges to the quotient of the extra cost $d$ for the green technology and the tax rate $b\alpha$. 

The remainder of this paper is organized as follows. Section~\ref{se:equilibrium} develops the model and its equilibrium. In Section~\ref{se:examples} we discuss special cases with closed-form solutions, giving first insights.  Section~\ref{se:compStat} presents the qualitative comparative statics.  The repeated game is discussed in Section~\ref{se:dynamics}, and Section~\ref{se:conclusion} concludes. Appendix~\ref{se:detailsEquilibrium} contains the proofs for Section~\ref{se:equilibrium} and a more detailed mathematical description of the equilibrium. Appendix~\ref{se:appendixExamples} elaborates on the examples of Section~\ref{se:examples}. In Appendix~\ref{se:compStatDetails} we derive quantitative comparative statics which imply, in particular, the qualitative comparative statics summarized in Section~\ref{se:compStat}. Appendix~\ref{se:appendixDynamics} contains the proofs for Section~\ref{se:dynamics}. Finally, Appendix~\ref{se:otherUtility} discusses more general utility functions for consumers.

\section{Equilibrium}\label{se:equilibrium}

Let $r_{i}\in [0,1]$ be the technology and $q_{i}\in \R_{+}$ the production quantity chosen by firm~$i$. The corresponding carbon emission is $k_{i}=r_{i}q_{i}$ where we choose units so that the business-as-usual technology $r=1$ corresponds to one unit of carbon per unit of good. Note that given $q_{i}$, we may equivalently specify $r_{i}$ or $k_{i}$.\footnote{The convention that $r_{i}=0$ when $q_{i}=0$ is used for the boundary case.} In addition to the emissions of the firms, we also include an exogenous amount $K_{ex}$ of carbon which may account for emitters outside the economy or pre-existing emissions. Thus, the total carbon is $K=K_{ex}+\sum_{j=1}^{n} r_{j}q_{j}$ and the total supply is $Q=\sum_{j=1}^{n} q_{j}$. 

To analyze the equilibrium, we first derive the optimality conditions for a fixed firm~$i$ given the quantity and carbon from sources other than firm~$i$, denoted $Q_{-i}=Q-q_{i}$ and $K_{-i}=K-k_{i}$. As mentioned in the Introduction, the net unit price (i.e., the revenue per unit of good net of taxes) for firm $i$ is
$$
  p_{i} = u'(Q) - bT\alpha r_{i}
$$
were $u: \R_{+}\to\R$ is a utility function, $\alpha$ is the TCRE and $T=\alpha K$ is the temperature increase. We assume that the choice $r=1$ with the highest emission has a unit cost of $c>0$ whereas the zero-emission technology bears a premium of $d>0$. The total production cost for a quantity $q$ at technology $r\in[0,1]$ is
$
  C(r,q)=[c + (1-r)d]q. %
$
In summary, the profit for the choices $(r_{i},q_{i})$ is
$$
  \pi_{i}(r_{i},q_{i}) = p_{i}q_{i} - C(q_{i},r_{i}) = u'(Q)q_{i} - b\alpha T r_{i}q_{i} - [c+(1-r_{i})d] q_{i}.
$$
Each firm $i$ has a belief about the distribution of $\alpha$. We denote by
$$
  \alpha_{i}^{2} = E_{i}[\alpha^{2}]
$$
the second moment of $\alpha$ under firm $i$'s belief. In most of the paper we endow consumers with the quadratic utility 
$u(x)=-\frac12 (A-x)^{2}$ for $x\in [0,A]$, where $A>0$ (and $u(x)=0$ for $x>A$);
see Appendix~\ref{se:otherUtility} for more general utility functions. In other words, the inverse demand $u'(x)=(A-x)^{+}$ is affine and the expected profit of firm $i$ under its subjective belief takes the form
\begin{equation}\label{eq:expectedProfitQuadratic}
E_{i}[\pi_{i}(r_{i},q_{i})] 
  =(A-q_{i}-Q_{-i})q_{i} - b\alpha_{i}^{2} (r_{i} q_{i} + K_{-i})r_{i} q_{i} - (c+d-dr_{i})q_{i}
\end{equation}
as long as $q_{i}+Q_{-i}\in[0,A]$. 
A (Nash) equilibrium is defined as a profile $(r_{j},q_{j})_{1\leq j\leq n}$ such that $(r_{i},q_{i})$ maximizes firm $i$'s expected profit~\eqref{eq:expectedProfitQuadratic} given $Q_{-i}=\sum_{j\neq i} q_{j}$ and $K_{-i}=K_{ex}+\sum_{j\neq i} r_{j}q_{j}$, for every $1\leq i\leq n$.

\begin{remark}\label{rk:onlySecondMoment}
  As the beliefs only affect the equilibrium through the expected profits~\eqref{eq:expectedProfitQuadratic}, the second moment $\alpha_{i}^{2}=E_{i}[\alpha^{2}]$  is a sufficient statistic for firm $i$'s views about~$\alpha$. The relation
  $
    E_{i}[\alpha^{2}]=\Var_{i}(\alpha)+E_{i}[\alpha]^{2}
  $
  shows that an increase in variance affects the equilibrium in the same manner as if the firm had a larger expected value: acknowledging uncertainty about $\alpha$ is equivalent to expecting a larger TCRE.
  
  In this spirit, one may replace~$b$ by a firm-dependent constant~$b_{i}$ which can be interpreted as a risk-aversion parameter, similarly as in Markowitz' problem. Setting~$\beta_{i}=b_{i}\alpha_{i}^{2}$ instead of $\beta_{i}=b\alpha_{i}^{2}$ in~\eqref{eq:abbreviations} below, the formulas in our results then continue to hold as stated.
\end{remark}

Given exogenous quantity $Q_{-i}$ and carbon $K_{-i}$, firm $i$ has a unique optimal choice $(r_{i},q_{i})$ which, however, is somewhat complicated to state because the choice is two-dimensional and subject to several constraints. We provide a detailed description in Appendix~\ref{se:detailsEquilibrium} and confine ourselves to an informal version in the main text, highlighting some of the key features. 
To facilitate the exposition we introduce the following color-coding. Firm $i$ is called \emph{white} if it does not produce ($q_{i}=0$). For the case of a positive production, we distinguish three cases: firm $i$ is \emph{green} if it produces exclusively with the emission-free technology ($r_{i}=0$ and $q_{i}>0$), \emph{red} if it produces exclusively with the business-as-usual technology ($r_{i}=1$), and \emph{orange} if uses an intermediate technology ($0<r_{i}<1$). 
The color captures which of the constraints (nonnegative production, technology between~$0$~and~$1$) are binding. The color itself depends on the belief, $Q_{-i}$ and $K_{-i}$, but once the color is determined, the optimal choice has a simple expression as stated below. The following definitions will be useful to obtain concise expressions, here and in the rest of the paper:
\begin{equation}\label{eq:abbreviations}
  \beta_{i}=b\alpha_{i}^{2}, \qquad a_{i}=\frac{d}{\beta_{i}}, \qquad z=A-c-d.
\end{equation}
Indeed, the second moment $\alpha_{i}^{2}$ of firm $i$'s belief on the TCRE can only occur via its product $\beta_{i}$ with the constant $b$ in the tax rate; cf.~\eqref{eq:expectedProfitQuadratic}. For green and orange firms, choices are tradeoffs between the extra cost $d$ of the green technology and $\beta_{i}$, which suggests the definition of $a_{i}$. A large value of $a_{i}$ corresponds to the view that taxes will be low or that mitigation is costly, so that firms with higher $a_{i}$ will make smaller mitigation efforts. Finally, the difference $z$ between the maximal demand $A$ and the unit cost $c+d$ of the green technology is clearly an important quantity for the mitigation efforts.

\begin{proposition}\label{pr:FOCbody}
  Given exogenous quantity $Q_{-i}$ and carbon $K_{-i}$, the color of firm $i$ is uniquely determined and the optimal choices are as follows. If firm $i$ is
\begin{enumerate}
  \item white, then $q_{i}=0$ and $k_{i}=0$.
  
  \item green, then $q_{i}=\frac12 (z - Q_{-i})$ and $k_{i}=0$.
  
  \item orange, then $q_{i}=\frac12(z - Q_{-i})$ and $k_{i}=\frac12(a_{i}-K_{-i})$.
  
  \item red, then
    $q_{i}=k_{i}=\frac12\frac{1}{1+\beta_{i}} [ A-c - Q_{-i} - \beta_{i} K_{-i}]$.
\end{enumerate}
\end{proposition} 

For all colors, $q_{i}$ and $k_{i}$ are weakly decreasing functions of $Q_{-i}$ and $K_{-i}$; that is, quantity and emissions act as substitutes. For white, green and orange firms, $q_{i}$ depends only on $Q_{-i}$ and $k_{i}$ depends only on $K_{-i}$. Whereas for red firms, $q_{i}$ and $k_{i}$ depend jointly on $Q_{-i}$ and $K_{-i}$, and moreover the precise coupling between the two depends on the specific belief of the firm in question. (In addition, the color of a firm depends on both $Q_{-i}$ and $K_{-i}$ and the firm's belief; cf.\ Appendix~\ref{se:detailsEquilibrium}.) 

\begin{theorem}\label{th:mainBody}
  There exists a unique equilibrium.
\end{theorem}

The proof of existence in Appendix~\ref{se:detailsEquilibrium} applies Brouwer's fixed point theorem in a fairly direct manner. Uniqueness is less obvious and the proof may be of interest on its own. While strategic substitutes generally imply uniqueness in the case of a one-dimensional control variable, this is not necessarily the case in a problem with two interacting controls---a priori, it may be possible to have an alternative equilibrium with smaller quantity but larger emissions. One key step in our proof is to exhibit a transversality relation between the color buckets (Lemma~\ref{le:antimonotone}): if green and red firms increase their production quantity, the orange firms would react by partially, but not fully, compensating that increase. Conversely, a change caused by orange firms would be over-compensated by the other firms.

Not all color combinations can arise in equilibrium: green and orange firms cannot co-exist with white ones; i.e., an equilibrium consist either of green, orange and red firms; or of white and red firms. (Some of these buckets may be empty; for instance, all firms can be orange.) As is intuitive, these colors are ordered in terms of climate beliefs: In the green-orange-red case, the green firms are the ones expecting the highest climate impacts (the highest taxes) and the red ones expect to lowest. In the white-red case, the white firms expect the higher impacts. See Appendix~\ref{se:detailsEquilibrium} for more details.

\section{Examples}\label{se:examples}

In this section we exhibit special cases with closed-form solutions that give more insight into the mechanics of the equilibrium.
The proofs boil down to verifying the optimality conditions for all firms; this is straightforward (and omitted) for Section~\ref{se:twoFirms}, whereas for Section~\ref{se:moderateDisagreement} we report proofs in Appendix~\ref{se:appendixExamples}.

\subsection{Two Firms}\label{se:twoFirms}

The case of two firms ($n=2$) is particularly simple as only two of the color buckets can be populated. For simplicity, we also assume that $K_{ex}=0$ and $z>d$---the latter eliminates the possibility of white firms; see Appendix~\ref{se:appendixExamples} for a more complete analysis.\footnote{The case $z\leq d$ gives rise to an additional regime (white-red) when one firm believes in high climate impacts and the other is very skeptical, and some additional restrictions in the other regimes.}   Without loss of generality, we label the firms such that $0\leq a_{1}\leq a_{2}$. 
Depending on the parameters $a_{1}$ and $a_{2}$, the equilibrium is in one of the six regimes listed below. These regimes are shown in Figure~\ref{fi:twoFirmsMain} above the diagonal ($a_{1}=a_{2}$), whereas the symmetric cases below the diagonal correspond to $a_{1}\geq a_{2}$. For instance, starting at the center of the diamond and moving north corresponds to fixing firm 1 and increasing $a_{2}$, meaning that firm~2 becomes more skeptical about climate impacts. As the heterogeneity increases, firm~2 reduces mitigation efforts and eventually abandons them (becomes red), but continues to increase emissions by increasing the production quantity. As a reaction, firm~1 increases mitigation efforts and eventually becomes green.

\begin{figure}[thb]
\begin{center}
\begin{tikzpicture}[scale=.82]%
\fill [color = myorange, opacity=0.9] (0,0) -- (2,4) -- (6,6) -- (4,2) -- cycle;
\fill [pattern = flexcheckerboard_greenorange, opacity=0.9] (0,0) -- (2,4) -- (0,4) -- cycle;
\fill [pattern = flexcheckerboard_greenorange, opacity=0.9] (0,0) -- (4,2) -- (4,0) -- cycle;
\fill [color = myred, opacity=0.9] (6,6) -- (6,8) -- (8,8) -- (8,6) -- cycle;
\fill [pattern = flexcheckerboard_greenred, opacity=0.9] (4,0) -- (4,2) --  plot [smooth,domain=2:2.91] ({20*\x/(16-3*\x)},\x) -- (8,0) -- cycle;
\fill [pattern = flexcheckerboard_orangered, opacity=0.9] (4,2) --  plot [smooth,domain=2:2.91] ({20*\x/(16-3*\x)},\x) -- (8,6) -- (6,6) -- cycle;

\fill [pattern = flexcheckerboard_greenred, opacity=0.9] (0,4) -- (2,4) --  plot [smooth,domain=2:2.91] (\x,{20*\x/(16-3*\x)}) -- (0,8) -- cycle;
\fill [pattern = flexcheckerboard_orangered, opacity=0.9] (2,4) --  plot [smooth,domain=2:2.91] (\x,{20*\x/(16-3*\x)}) -- (6,8) -- (6,6) -- cycle;

\draw[->] (0,0) -- (8.5,0) node[right] {$a_1$};
\draw[->] (0,0) -- (0,8.5) node[above] {$a_2$};

\draw (0,0) -- (4,2) -- (6,6) {};
\draw (0,0) -- (2,4) -- (6,6) {};

\draw (4,2) -- (4,-0.1) node[below] {$\tfrac{2}{3}z$};
\draw (2,4) -- (-0.1,4) node[left] {$\tfrac{2}{3}z$};

\draw (6,6) -- (6,8) {};
\draw (6,6) -- (8,6) {};

\draw [smooth,domain=2:2.9] plot (\x,{20*\x/(16-3*\x)});
\draw [smooth,domain=2:2.9] plot ({20*\x/(16-3*\x)},\x);

\draw (6,0) -- (6,-0.1) node[below] {$z$};
\draw (0,6) -- (-0.1,6) node[left] {$z$};

\draw (2,0) -- (2,-0.1) node[below] {$\tfrac{z}{3}$};
\draw (0,2) -- (-0.1,2) node[left] {$\tfrac{z}{3}$};

\end{tikzpicture}
\end{center}
\caption{Regimes for the equilibrium with $n=2$ firms. Coloring represents the colors of the two firms in the respective regime.}
\label{fi:twoFirmsMain}
\end{figure}
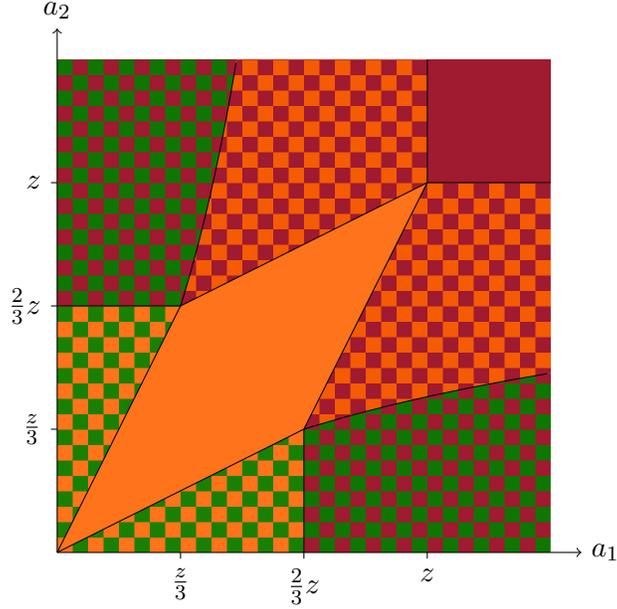

\begin{enumerate}[(a)]
\item \emph{Orange-orange.} Suppose that $a_{1}>a_{2}/2$ and $a_{2}<(z+a_{1})/2$. Then both firms are orange, 
$$
  Q=\frac{2z}{3},\quad K = \frac{a_{1}+a_{2}}{3},\quad q_{1}=q_{2}=\frac{z}{3}, \quad k_{1}=\frac{2a_{1}-a_{2}}{3},\quad k_{2}=\frac{2a_{2}-a_{1}}{3}.
$$
This regime is ``interior'' in that no constraint is binding. It arises when the coefficients $a_{1}$ and $a_{2}$ are neither too small nor too large and moreover the heterogeneity (i.e., the fraction $a_{2}/a_{1}$) is not too large.

\item \emph{Green-orange.} Suppose that $a_{1}\leq a_{2}/2$ and $a_{2}<2z/3$. Then firm 1 is green and firm 2 is orange. We have
$$
  Q=\frac{2z}{3},\quad K = \frac{a_{2}}{2},\quad q_{1}=q_{2}=\frac{z}{3}, \quad k_{1}=0,\quad k_{2}= \frac{a_{2}}{2}.
$$
This equilibrium is similar to the previous one but firm 1 expects climate impacts so large that it only uses emission-free technology. The quantities $q_{j}$ remain identical and all expressions remain affine.

\item \emph{Red-red.} Suppose that $a_{1}\geq z$. Then both firms are red and
$$
  Q=K=\frac{A-c}{3}\Big(\frac{1}{1+\beta_{1}}+\frac{1}{1+\beta_{2}}\Big), \quad q_{1}=k_{1}=\frac{A-c}{3}\Big(\frac{2}{1+\beta_{1}}-\frac{1}{1+\beta_{2}}\Big)
$$
and symmetrically $q_{2}=k_{2}=\frac{A-c}{3}\big(\frac{2}{1+\beta_{2}}-\frac{1}{1+\beta_{1}}\big)$. In this regime neither firm is sufficiently incentivized to use mitigate emissions. Individual as well as aggregate quantities depend explicitly on the beliefs of both firms.

\item \emph{Green-red.} Suppose that $a_{1}\leq \frac{(z+2d)a_{2}}{3a_{2}+4d}$ and $a_{2}\geq 2z/3$. Then firm 1 is green and firm 2 is red. We have $k_{1}=0$ and 
$$
  Q=\frac{2(1+\beta_{2})z+d}{3+4\beta_{2}}, \quad K=\frac{z + 2d}{3 + 4\beta_2}=q_{2}=k_{2}, \quad q_{1}=\frac{(1+2\beta_2)z - d}{3 + 4\beta_2}.
$$
This is the regime of extreme disagreement, firm 1 is emission-free whereas firm 2 makes no effort to reduce carbon. The formulas depend on the belief of the red firm (firm 2). For any given value of $a_{1}$, Firm~1 will be green if Firm~2 is sufficiently skeptical.

\item \emph{Orange-red.} Suppose that $\frac{(z+2d)a_{2}}{3a_{2}+4d}<a_{1}<z$ and $a_{2}\geq (z+a_{1})/2$. Then firm 1 is orange and firm 2 is red. We have
$$
Q = \frac{A-c - d\frac{a_1}{2a_2}-z/2}{3(1+\beta_2)}+\frac{z}{2},\quad K=Q+\frac{a_{1}-z}{2}, \quad q_{1}=z-Q,\quad q_{2}=k_{2}=2Q-z
$$
and $k_{1}=\frac{a_{1}+z}{2}-Q$.
This a regime of intermediate disagreement where firm~1 makes some effort but firm~2 makes no effort to reduce carbon. The formulas depend on the beliefs of both firms.

\item \emph{Green-green.} The corner equilibrium $a_{1}=a_{2}=0$ corresponds to the limiting case where both firms fear infinite climate impacts. Both firms emit zero carbon while producing the common quantity $q_{1}=q_{2}=z/3$. This regime would occur for a larger range of coefficients if we had allowed for exogenous carbon, $K_{ex}>0$.
\end{enumerate}

\begin{remark}\label{rk:boundaries}
As visualized in Figure~\ref{fi:twoFirmsMain}, equilibria exist in different regimes depending on the parameter values. Each regime is a subsets of $\R^{n}$ and its boundary is a piecewise smooth hypersurface of codimension 1. For our results on equilibria, it does not matter if the boundary between two regimes is seen as part of one or the other regime. For instance, if the given parameters $(a_{1},a_{2})$ are on the boundary between the orange-orange and the green-orange regime, we may see the equilibrium as part of either regime---the formulas stated in those regimes give the same result for such $(a_{1},a_{2})$.
Mathematically, the statements about the regimes are continuous and hence remain valid on the closure. This holds true for general equilibria with any number of firms.
\end{remark}

\subsection{Moderate Disagreement}\label{se:moderateDisagreement}

The following examples discuss $n$-player equilibria which are particularly tractable because either all firms make some effort to reduce carbon or no firm does. These cases exhibit at most moderate heterogeneity between the firms. %
The complexity of the equilibrium increases substantially if the constraint $r\leq1$ is binding for some (but not all) firms, as can already be seen in the above case of two firms---cf.\ the green-red and orange-red regimes.

The simplest $n$-player equilibrium arises when none of the constraints is binding; i.e., all firms are orange. Then, firms adjust for their belief through the technology choice but produce a common quantity independent of the belief. This situation occurs when the heterogeneity is sufficiently small and the coefficients are neither too small nor large. We also assume $K_{ex}=0$ to further simplify the expressions; this is not crucial.

\begin{example}[Orange]\label{ex:allInt}
  Let $z>0$ and $a_{1}\leq \cdots\leq a_{n}$ and $K_{ex}=0$. Suppose that
  \begin{equation}\label{eq:condForAllInt}
    a_{1} \geq \avg\{0,a_{2},\dots,a_{n}\}\quad \mbox{and}\quad  a_{n} \leq \avg\{z,a_{1},\dots,a_{n-1}\}.
  \end{equation}
  Then the equilibrium satisfies
  $$
    Q = \frac{nz}{n+1},\quad K= \frac{1}{n+1} \sum_{j=1}^{n} a_{j}, \quad q_{i} = \frac{z}{n+1},\quad r_{i}= \frac{1}{z}\Bigg(na_{i}-\sum_{j\neq i}a_{j}\Bigg)
  $$
  for all $1\leq i\leq n$. A sufficient condition for~\eqref{eq:condForAllInt} is that $a_{1}\geq \frac{n-1}{n}a_{n}$ and $a_{n}\leq \frac{z}{n}+\frac{n-1}{n}a_{1}$.
\end{example}

Example~\ref{ex:allInt} is a special case of the following situation where the constraint $r\geq0$ may be binding but $r\leq1$ is not. This arises when none of the firms is much more skeptical about climate change than the others, and preserves the crucial feature of Example~\ref{ex:allInt}; namely, that firms adjust their technology $r_{i}$ to account for carbon whereas the quantities $q_{i}$ are unaffected by carbon emissions and beliefs. In the subsequent example, there are $n_{0}$ green firms and $m=n-n_{0}$ orange firms, and the number $n_{0}$ is determined analytically from the beliefs.

\begin{example}[Green-orange]\label{ex:noneAtOne}
  Let $z>0$ and $a_{1}\leq \cdots\leq a_{n}$ and $K_{ex}\geq0$. Define\footnote{\label{fo:max}For the definition of $n_{0}$ we use the convention $\max\emptyset=0$.}
  $$
    n_{0}= \max\Bigg\{i:\, a_{i}< K_{ex} + \sum_{j=i+1}^{n} (a_{j}-a_{i})\Bigg\}, \quad m=n-n_{0}, \quad A_{m}=\sum_{j=n_{0}+1}^{n} a_{j}.
  $$
  Moreover, suppose that
  \begin{equation}\label{eq:condNoneAtOne}
    a_{n} \leq \frac{z}{n+1} + \frac{A_{m}+K_{ex}}{m+1}.
\end{equation}
  Then the equilibrium satisfies\footnote{In fact, \eqref{eq:condNoneAtOne} is not only sufficient but also necessary for absence of red firms.}  
  $$
    Q = \frac{nz}{n+1}, \quad K=\frac{A_{m}+K_{ex}}{m+1}, \quad q_{i} = \frac{z}{n+1}, \quad 1\leq i\leq n
  $$
  as well as $r_{i}=0$ for $1\leq i\leq n_{0}$ and 
  $r_{i}= \frac{n+1}{z} \big( a_{i} - \frac{A_{m}+K_{ex}}{m+1} \big)$ for $n_{0}<i\leq n$.
  A sufficient condition for~\eqref{eq:condNoneAtOne} is that $a_{n}\leq \frac{2z}{n+1}+K_{ex}$. 
\end{example}

Equilibria are more complicated when the constraint $r\leq 1$ is binding for at least one firm: while the green and orange firms continue to produce a common quantity, that quantity is now influenced by the views of the red firms, and each red firm may have a different quantity. This  precludes simple closed-form solutions in most cases. An exception arises when all firms are sufficiently skeptical: in the following example, the number $n_{0}$ of white firms (which cease production completely) is determined analytically from the beliefs.

\begin{example}[White-red]\label{ex:noGreenTech}
  Let $A>c$ and $a_{1}\leq \cdots\leq a_{n}$ and $K_{ex}\geq0$. Define\footnote{Footnote~\ref{fo:max} applies.} 
  $$
    \xi_{j}= \frac{A-c -b\alpha_{j}^{2}K_{ex}}{1+b\alpha_{j}^{2}}, \quad n_{0}= \max\Bigg\{i:\, \xi_{i} < \sum_{j=i+1}^{n} (\xi_{j}-\xi_{i})\Bigg\}, \quad n_{1}=n-n_{0}.
  $$  
  Suppose that $a_{1}\geq z + K_{ex}$. Then the equilibrium satisfies 
  $Q = \frac{1}{n_{1}+1}\sum_{j > n_{0}} \xi_{j}$ and $K=Q+K_{ex}$,
  as well as $q_{i}=k_{i}=0$ for $i\leq n_{0}$ and 
  $q_{i} = k_{i}=\frac{n_{1}}{n_{1}+1}\xi_{i}
     - \frac{1}{n_{1}+1} \sum_{i\neq j > n_{0}} \xi_{j}$ for $i > n_{0}$.
\end{example}

\section{Comparative Statics}\label{se:compStat}

In this section we analyze how a change in a firm's view impacts the firm's decisions, its competitors and the overall economy. 
The subjective second moment $\alpha_{j}^{2}=E_{j}[\alpha^{2}]$ of the TCRE is called the \emph{climate belief} (or simply belief) of firm~$j$; cf.\ Remark~\ref{rk:onlySecondMoment}. An increase in belief corresponds to higher expected climate impacts/taxes  whereas a decrease corresponds to the firm becoming more skeptical. %
For simplicity of exposition we assume that there are no firms with zero production quantity (white firms)---in any event, such firms do not directly affect the rest of the economy. Moreover, no firm is infinitely skeptical (i.e., $\alpha_{j}^{2}>0$ for all~$j$). Thus, firms are green (produce exclusively with zero-emission technology), orange (emit carbon with some effort to reduce emissions) or red (business-as-usual). The statements below are valid for perturbations of the climate belief that keep the equilibrium in the same regime; that is, firms do not change color during the perturbation. This is always true if the perturbation is sufficiently small.\footnote{If the initial equilibrium is on the boundary between two regimes, we use the flexibility mentioned in Remark~\ref{rk:boundaries} and define the boundary as part of a regime which is preserved by the perturbation. Due to the differentiability of the boundaries, this is always possible.} The three theorems below summarize most of the qualitative insights; they are corollaries of the more detailed, quantitative results reported in Appendix~\ref{se:compStatDetails}.
We start with the comparative statics for the overall economy.

\begin{theorem}\label{th:compStatsTotal}
  The total carbon emission $K$ is weakly decreasing in the climate beliefs of all firms:
  \begin{enumerate}
  \item $K$ is unaffected by the beliefs of green firms.
  \item $K$ is strictly decreasing in the beliefs of all other firms.
  \end{enumerate}
  The dependence of the total production quantity $Q$ is ambiguous:
  \begin{enumerate}
  \item $Q$ is unaffected by the beliefs of green firms.
  \item $Q$ is strictly decreasing in the beliefs of red firms.
  \item $Q$ is strictly increasing in the beliefs of orange firms---except if there are no red firms in the economy, 
   in which case $Q$ is unaffected by the beliefs of orange firms.
  \end{enumerate}  
\end{theorem}

The sensitivities for the total carbon are intuitive: if some firm emits carbon and becomes more concerned about climate impacts, it will reduce its emissions. As we will see below, other firms may increase their emissions in response, but the overall effect is still a reduction.

For the production quantity the situation is more complex. If a red firm becomes more concerned about climate impacts, it will reduce its emissions---and hence its quantity, as these are equal for red firms. Other firms may react with an increased production (see below), but again the overall effect is a reduction. If an orange firm becomes more concerned about climate impacts, then initially (more precisely, neglecting feedback effects from equilibrium) it would reduce its emissions by changing to a greener technology but keep its quantity constant. While other orange firms react by slightly increasing their emissions, the collection of all orange firms would still emit less in total, and leave the quantity unchanged. Red firms, however, now face an environment with lowered carbon and similar quantities from their competitors, thus increase their production, and consequently emissions. This increase is large enough to over-compensate the reduction in quantity from the orange firms (whereas the overall carbon is still reduced, as seen above). The exception %
arises when there are no red firms present to carry out this mechanism.\footnote{The same happens if all red firms are infinitely skeptical and thus completely unaffected by emissions, a situation that was excluded in this section.}

Next, we turn to the dependence of a firm's choice on its own belief and the beliefs of other firms. A clear-cut result holds for the technology choices which behave like strategic substitutes, with a strict monotonicity unless the firm is subject to binding constraints.

\begin{theorem}\label{th:compStatsTech}
  Consider the equilibrium technology choice $r_{i}$ of any firm $i$.
  \begin{enumerate}
  \item $r_{i}$ is weakly decreasing in $i$'s own belief. The decrease is strict iff~$i$ is orange.\footnote{Here ``iff'' stands for ``if and only if.''} 
  \item $r_{i}$ is weakly increasing in the belief of any other firm $j\neq i$. The increase is strict iff~$i$ is orange and~$j$ emits carbon (i.e., $j$ is not green).
  \end{enumerate}
\end{theorem}

In the preceding result, binding constraints cause little complication because the technology choice remains constant at those boundaries. This is not the case for the production quantities and carbon emissions, whose comparative statics depend on the type of firm. %

\begin{theorem}\label{th:compStatsRest}
  (a) Consider the equilibrium carbon emission $k_{i}$ and production quantity $q_{i}$ of any firm $i$.
  \begin{enumerate}
  \item $k_{i}$ is weakly decreasing in firm $i$'s own belief. The decrease is strict unless $i$ is green.
  \item $q_{i}$ is weakly decreasing in firm $i$'s own belief. The decrease is strict as long as $i$ is not green and red firms exist in the economy.
  \end{enumerate}
  (b) Consider a second firm $j\neq i$.
  \begin{enumerate}
  \item If firm $j$ is green, its belief does not affect other firms.
  \item If firm $j$ is orange and firm $i$ is red, $k_{i}$ and $q_{i}$ depend ambiguously on $j$'s belief. The direction depends on the other firms (see Remark~\ref{rk:ambiguousQuantity}).
  \end{enumerate}
  In the remaining cases,
    \begin{enumerate}
    \item[(iii)] $k_{i}$ is weakly increasing in $j$'s belief, and strictly increasing unless $i$ is green,
  \item[(iv)]  $q_{i}$ is strictly increasing in $j$'s belief if $j$ is red, but weakly decreasing if $j$ is orange. The decrease is strict unless there are no red firms.
  \end{enumerate}
\end{theorem}

The results on $k_{i}$ are mostly intuitive. If a firm $j$ grows more concerned about climate impacts, it reduces its emissions. As a consequence, the marginal cost of carbon decreases and other firms increase their emissions. The case where $i$ is red and $j$ is orange is more complex, in part because the carbon is coupled with the production quantity for red firms. The direction of change then depends on the characteristics of other red firms (if any) and the ranking of the beliefs among the red firms; cf.\ Remark~\ref{rk:ambiguousQuantity} for details.

The observations about $q_{i}$ can be understood as in the discussion after Theorem~\ref{th:compStatsTotal}. If $j$ is red and grows more concerned about climate impacts, it reduces its emissions and hence, a fortiori, its production quantity. Other firms then face less competition and increase their quantity. If $j$ is orange, however, it reduces emissions by changing technology rather than quantity. Red firms then increase their production since the marginal cost of carbon has decreased, but green and orange firms decrease their production as a reaction to the competition from red firms. Again, the case (ii) is discussed in Remark~\ref{rk:ambiguousQuantity}.

\section{Repeated Game and Cumulative Temperature Change}\label{se:dynamics}

In this section we consider a repeated version of the Cournot game. Suppose that after the carbon has been emitted and all goods have been sold,  firms start a new planning phase similar to the first one. The emitted carbon, zero before the first round, accumulates and becomes external carbon for the next round. The goods from the previous rounds are considered consumed, so that the demand is determined by the same utility function in each round. %
We assume that firms are myopic in their planning, but \MN{firms' views on climate change can evolve from one round to the next:} the coefficients in the $m$-th round are denoted $a^{(m)}_{j}$. \MN{Under specific conditions, we will see below} that the long-run limit $m\to\infty$ allows for closed-form expressions for the accumulated carbon and temperature increase.

\MN{The scenario we have in mind is that the ``true'' value of the TCRE is approximately constant over time and firms learn this value as time progresses, so that all $a^{(m)}_{j}$ converge to the same value as $m\to\infty$. Such a scenario is more likely when total emissions remain small. 
If the TCRE increases substantially over time or the observed climate changes dramatically so that firms' optimization problems are altered fundamentally (e.g., a climate tipping event is observed and triggers changes in regulation and firms' available actions), the modeling as a repeated game as well as the technical condition below do not hold.}

The following results show that the total carbon \MN{in the repeated game} stabilizes at a level that is determined by the most skeptical firm in the long run; i.e., the minimal parameter~$\alpha_{i}^{2}$ or equivalently the maximal~$a_{i}$. In our next result, firms increasingly use green technology and the total carbon stabilizes at the level of the largest limit point
$$
  a:=\limsup_{m\to\infty}\max\{a^{(m)}_{1},\dots,a^{(m)}_{n}\} \in[0,\infty].
$$
\MN{The result assumes that $a^{(m)}_{j}\leq a$ for all $m$ and $j$. }
This is clearly satisfied if the sequences $a^{(m)}_{j}$ are increasing in $m$, as would be the case e.g.\ if the variance of the TCRE under the subjective views decreases over time while the mean is constant. \MN{Importantly, the assumption excludes a scenario where some coefficients $a^{(m)}_{j}$ are high at an intermediate time but all coefficients eventually become small. Then, carbon at the intermediate time may exceed $a$ even though it accumulated in a relatively shorter time---after the intermediate period all firms become so concerned about climate impacts that they use the zero-emissions technology. An obvious example is when $a=0$ and $a^{(1)}>0$. For instance, suppose that in the first rounds, some firms believe that a tipping event will happen and others do not. If at some point the tipping event indeed happens and all firms use the zero-emissions technology going forward, the total carbon would remain at the threshold level of the event rather than corresponding to a long run belief.}

\begin{proposition}\label{pr:dynamics}
  Suppose that $c+d<A$ and $a^{(m)}_{j}\leq a$ for all $m$ and $j$. Then the accumulated carbon emissions converge to $a$ as $m\to\infty$.
\end{proposition}

  Suppose that the limit $a$ corresponds to the true (deterministic) TCRE $\alpha$. If $K=a$ denotes the limiting total carbon and~$T=\alpha K =\alpha a$ the corresponding temperature increase, Proposition~\ref{pr:dynamics} shows that 
  $$
    T=\frac{d}{b\alpha}.
  $$
  Recall the tax interpretation of the price function: $b\alpha$ is the tax rate, per unit of carbon emitted and temperature increase. Thus, the limiting temperature change $T$ is succinctly described as the quotient of the extra cost $d$ for the green technology and the tax rate.

The assumption that $c+d<A$ is essential in Proposition~\ref{pr:dynamicsNoGreen} because it allows firms to use green technology to reduce carbon emissions while keeping the quantity produced above a threshold. If $A \leq c+d$, consumers will not pay for the green technology and the limit is different: the total carbon now stabilizes because the production tends to zero and the economy comes to a standstill. We also assume that $c< A$; otherwise no goods are produced and the result is trivial.

\begin{proposition}\label{pr:dynamicsNoGreen}
Suppose that $c< A \leq c+d$  and $\beta^{(m)}_{j}\geq \beta$ for all $m$ and~$j$, where $\beta=\liminf_{m\to\infty} \min\{\beta^{(m)}_{1},\dots,\beta^{(m)}_{n}\} \in [0,\infty]$. 
Then the accumulated carbon emissions converge to $(A-c)/\beta$.
\end{proposition}

\section{Conclusion}\label{se:conclusion}
We formulate a partial equilibrium model where firms make irreversible decisions about production and emissions with the aim of maximizing expected future profits. Profits are reduced by carbon taxes at a rate that depends on future climate change, hence is endogenous and uncertain at the time of planning. Taxes are imposed by an outside regulator and incentivize firms to mitigate emissions.
Firms agree-to-disagree about the climate impact of carbon and therefore about the tax rate. 
The framework of agreeing to disagree seems adequate given that the equilibrium depends only on the second moments of the beliefs and actions (e.g., starting to build a nuclear plant) are mostly observable. This argument does not extend to the regulator, who is not part of the partial equilibrium. It may be interesting to study a model where the role of signaling for the regulator can be investigated.
In the present model, the regulator has already put in place an adjustment policy to a single source of uncertainty, the change of temperature. In reality, the regulator is concerned with the ``cost'' of climate change---which is itself uncertain and whose estimate changes with scientific advances and public opinion, creating an intricate and time-inconsistent decision problem \citep[e.g.,][]{ulph2013optimal,jakob2014optimal}.

Our model allows us to study how a firm would position itself in an economy where competitors differ in their expectations about the future cost of carbon. More generally, this may inform our thinking regarding changes in consumer preferences or other climate-related risks. In this model, mitigation efforts act as substitutes and are increasing in the variance of the subjective belief on the carbon-climate response. That is, for a  given firm, having skeptical competitors and large uncertainty leads to higher mitigation efforts. 
This is consistent with a standard Cournot model without technology choice and taxes, yet a more detailed analysis of the comparative statics reveals that reactions in terms of production and carbon quantity depend on the relative position of the firm in the economy. Indeed, the technology choice decouples production and emissions and hence allows firms to  react differently to the prices of the good and carbon. In several equilibrium regimes, carbon taxes reduce emissions without lowering the production quantity.

\newcommand{\dummy}[1]{}

\appendix

\section{Existence, Uniqueness, Characterization of Equilibrium}\label{se:detailsEquilibrium}

The following result states the optimality conditions for a given firm. In particular, these formulas apply in any equilibrium. The sets $I_{0}^{q}, I_{0}^{r}, I_{int}, I_{1}$ in the proposition correspond to the color coding white, green, orange, red used in the body of the text. We recall the quantities introduced in~\eqref{eq:abbreviations}.

\begin{proposition}\label{pr:FOC}  
  Let $A,c,d,\alpha_{i}^{2}>0$ and $K_{-i}\geq0$ and $0\leq Q_{-i} \leq A$. Set $z=A-c-d$ and $\beta_{i}=b\alpha_{i}^{2}$ and $a_{i}=d/\beta_{j}$.\footnote{In the case $\alpha_{i}^{2}=0$ the statements need to be read with $\beta_{i}=0$, $a_{i}=\infty$ and $\beta_{i}a_{i}=d$.} Define the sets 
  \begin{align*}
      I_{0}^{q}&=\{i:\, z-Q_{-i}+\beta_{i}(a_{i}-K_{-i})\leq0 \mbox{ and } Q_{-i}\geq z\},\\
      I_{0}^{r}&=\{i:\, K_{-i}\geq a_{i} \mbox{ and } Q_{-i}< z\},\\
      I_{int}&=\{i:\, K_{-i}< a_{i} \mbox{ and } Q_{-i}-K_{-i}<z-a_{i}\},\\
      I_{1}&=\{i:\, z-Q_{-i}+\beta_{i}(a_{i}-K_{-i})>0 \mbox{ and } Q_{-i}-K_{-i}\geq z-a_{i} \}.
  \end{align*}   
  These sets form a partition of $\{1,\dots,n\}$.
  Fix a firm $i$ and suppose the quantity $Q_{-i}$ of the good and $K_{-i}$ carbon are supplied exogenously. Then there exists a response $(r_{i},q_{i})\in[0,1]\times [0,A-Q_{-i}]$ which maximizes the expected profit~\eqref{eq:expectedProfitQuadratic} of firm~$i$, and $(r_{i},q_{i})$ is unique with the convention that $r_{i}=0$ when $q_{i}=0$.
  Denoting $Q=Q_{-i}+q_{i}$ and $K=K_{-i}+r_{i}q_{i}$, we also have
  \begin{align*}
      I_{0}^{q}&=\{i:\, z-Q+\beta_{i}(a_{i}-K)\leq0 \mbox{ and } Q\geq z\},\\
      I_{0}^{r}&=\{i:\, K\geq a_{i} \mbox{ and } Q< z\},\\
      I_{int}&=\{i:\, K< a_{i} \mbox{ and } Q-K<z-a_{i}\},\\
      I_{1}&=\{i:\, z-Q+\beta_{i}(a_{i}-K)>0 \mbox{ and } Q-K\geq z-a_{i} \}.
  \end{align*} 
  Moreover, with $k_{i}=r_{i}q_{i}$, the following hold. 
\begin{enumerate}
  \item $i\in I^{q}_{0}$ if and only if $q_{i}=0$. Then, $k_{i}=0$ and $r_{i}=0$.
  
  \item $i\in I^{r}_{0}$  if and only if $r_{i}=0$ and  $q_{i}>0$. Then, $q_{i}=\frac12 (z - Q_{-i})=z-Q$ and $k_{i}=0$.

  \item $i\in I_{int}$  if and only if  $r_{i}\in(0,1)$ and $q_{i}>0$. Then, 
  $$
    q_{i}=\frac12(z - Q_{-i})=z - Q, \quad\quad\! k_{i}=\frac12(a_{i}-K_{-i})=a_{i}-K,
  \qquad \!
    r_{i}=\frac{a_{i}-K_{-i}}{z - Q_{-i}} = \frac{a_{i}-K}{z - Q}.
  $$
  \item $i\in I_{1}$  if and only if  $r_{i}=1 \mbox{ and } q_{i}>0$. Then, $k_{i}=q_{i}$ and
  $$
    q_{i}=\frac12\frac{1}{1+b\alpha_{i}^{2}} [ A-c - Q_{-i} - b\alpha_{i}^{2} K_{-i}]=\frac{1}{1+b\alpha_{i}^{2}} [ A-c - Q - b\alpha_{i}^{2} K].
  $$
\end{enumerate}
\end{proposition}

\begin{proof}
  We derive the assertions referring to $K_{-i}$ and $Q_{-i}$. Once these are established, the assertions referring to $K$ and $Q$ are a direct consequence. One verifies that $I^{q}_{0}$, $I^{r}_{0}$, $I_{int}$ and $I_{1}$ are disjoint and that their union is $\{1,\dots,n\}$, by using the inequalities in their definitions. We first assume that $\alpha_{i}^{2}>0$.
  
  Recall from~\eqref{eq:expectedProfitQuadratic} that for an arbitrary choice $(r,q)\in [0,1]\times [0, A - Q_{-i}]$, firm $i$'s expected profit is
  $$
E_{i}[\pi_{i}(r,q)] 
  =(A-q-Q_{-i}) q - b\alpha_{i}^{2} (r q + K_{-i})r q - (c+d-dr)q.
  $$
  We may express this in terms of $q$ and $k=rq$ as
  $$
  (A-q-Q_{-i}) q - b\alpha_{i}^{2}k^{2}- b\alpha_{i}^{2}K_{-i} k - (c+d)q-dk.
  $$  
  This continuous function is jointly strictly concave on the compact simplex $0\leq k\leq q\leq A - Q_{-i}$; therefore, it admits a unique maximizer. In view of our convention that $r_{i}=0$ as soon as $q_{i}=0$, it follows that $E_{i}[\pi_{i}(r,q)]$ has a unique maximizer $(q_{i},r_{i})\in [0,1]\times [0, A - Q_{-i}]$. In fact, as $q\mapsto E_{i}[\pi_{i}(r,q)]$ is strictly decreasing for $q\geq A - Q_{-i}$, we see that $(r_{i},q_{i})$ is also the unique maximizer in $[0,1]\times\R_{+}$. It will be convenient to rearrange the terms,
  $$
    E_{i}[\pi_{i}(r,q)] = [z-Q_{-i} + \beta_{i}(a_{i}-K_{-i})r]q - (1+\beta_{i} r^{2})q^{2}.
  $$
  Next, we analyze the first-order conditions for interior maxima as well as the potentially binding constraints $q\geq 0$ and $r\geq 0$ and $r\leq 1$.\\[-.5em]
  
  \emph{Case 1:} Suppose that $q_{i}=0$. Then $E_{i}[\pi_{i}(r_{i},q_{i})]=0=E_{i}[\pi_{i}(r,q_{i})]$ for all $r\in[0,1]$  and it follows that 
  $$
    0\geq\partial_{q}E_{i}[\pi_{i}(r,q)]|_{q=0}=z-Q_{-i}+\beta_{i}(a_{i}-K_{-i})r
  $$
  for all $r\in[0,1]$, as otherwise $(r_{i},q_{i})$ would not be a maximizer. The above holds in particular for $r=0$ and $r=1$; that is,
  $z-Q_{-i}\leq0$ and $z-Q_{-i}+\beta_{i}(a_{i}-K_{-i})\leq0$, or equivalently $i\in I^{q}_{0}$. We have $r_{i}=0$ by our convention and $k_{i}=0$ is clear.\\[-.5em]
  
  For the remaining cases, suppose that $q_{i}>0$. Then $q_{i}$ is an interior maximum and $\partial_{q}E_{i}[\pi_{i}(r_{i},q_{i})]=0$ which yields that $q_{i}=q_{i}(r_{i})$ for 
  $
    q_{i}(r)=\frac12\frac{1}{1+\beta_{i} r^{2}} [ z - Q_{-i} + \beta_{i} (a_{i}-K_{-i})r].
  $
  Moreover, for any $r\in[0,1]$, we have
  $
    E_{i}[\pi_{i}(r,q_{i}(r))] = \frac{[ z - Q_{-i} + \beta_{i} (a_{i}-K_{-i})r]^{2}}{4(1+\beta_{i} r^{2})}
  $
  and
  \begin{align*}
    \partial_{r}E_{i}[\pi_{i}(r,q_{i}(r))] 
    & = \frac{ z - Q_{-i} + \beta_{i} (a_{i}-K_{-i})r}{2(1+\beta_{i} r^{2})^{2}} \;\times \\[-.1em]
    & \quad\; \left\{ \beta_{i} (a_{i}-K_{-i}) (1+\beta_{i} r^{2}) - [ z - Q_{-i} + \beta_{i} (a_{i}-K_{-i})r] \beta_{i} r \right\}\\[.3em]
    &=\frac{\beta_{i} [ z - Q_{-i} + \beta_{i} (a_{i}-K_{-i})r]}{2(1+\beta_{i} r^{2})^{2}}
       \left\{a_{i}-K_{-i} - (z - Q_{-i}) r \right\}.
  \end{align*}
  
  \emph{Case 2:} Suppose that $r_{i}=0$ and $q_{i}>0$. Then 
  $
    0<q_{i}=q_{i}(0)= \frac12(z - Q_{-i})
  $
  and in particular $Q_{-i}<z$. Moreover, we must have 
  $
    0\geq \partial_{r}E_{i}[\pi_{i}(r,q_{i}(r))]|_{r=0}=\frac12 (z-Q_{-i})\beta_{i}(a_{i}-K_{-i})
  $
  which then yields $a_{i}\leq K_{-i}$. Thus, $i\in I^{r}_{0}$.\\[-.5em]
  
  \emph{Case 3:} Suppose that $r_{i}\in(0,1)$ and $q_{i}>0$. Then 
  $$
    0<q_{i}=q_{i}(r_{i})= \frac12\frac{1}{1+\beta_{i} r_{i}^{2}} [ z - Q_{-i} + \beta_{i} (a_{i}-K_{-i})r_{i}],
  $$  
  so that $z - Q_{-i} + \beta_{i} (a_{i}-K_{-i})r_{i}>0$ and at least one of the terms $z - Q_{-i}$ and $a_{i}-K_{-i}$ must be strictly positive. Now $\partial_{r}E_{i}[\pi_{i}(r_{i},q_{i}(r_{i}))]= 0$ yields  
  $
    a_{i}-K_{-i} = (z - Q_{-i}) r_{i}
  $
  and it follows both terms are positive. Moreover, $r_{i}=\frac{a_{i}-K_{-i}}{z-Q_{-i}}\in(0,1)$ shows that $0<a_{i}-K_{-i}<z-Q_{-i}$ or equivalently $i\in I_{int}$, and finally using the same formula for $r_{i}$ in the general expression for $q_{i}(r)$ also yields $q_{i}=q_{i}(r_{i})=\frac12(z - Q_{-i})$.\\[-.5em]
  
  \emph{Case 4:} Suppose that $r_{i}=1$ and $q_{i}>0$. Then
  $$
    0\leq \partial_{r}E_{i}[\pi_{i}(r,q_{i}(r))]|_{r=1} = 
    \frac{\beta_{i} [ z - Q_{-i} + \beta_{i} (a_{i}-K_{-i})]}{2(1+\beta_{i})^{2}}
       \!\left\{a_{i}-K_{-i} - z + Q_{-i} \right\}
  $$
  where once again $z - Q_{-i} + \beta_{i} (a_{i}-K_{-i})>0$ by the above formula for $q_{i}(r)$ and the assumption that $q_{i}>0$, so it follows that
  $a_{i}-K_{-i} \geq z - Q_{-i}$ and $i\in I_{1}$.
  Moreover,
  $$
    q_{i}(1) = \frac12\frac{1}{1+\beta_{i}} [ z - Q_{-i} + \beta_{i} (a_{i}-K_{-i})]= \frac12\frac{1}{1+\beta_{i}} [ A-c - Q_{-i} - \beta_{i} K_{-i})]
  $$
  after recalling that $z=A-c-d$ and $\beta_{i} a_{i}=d$.\\[-.5em]
  
  Finally, note that in view of the partition property and the fact that we have discussed all possible cases for $r_{i}$ and $q_{i}$, the above implications show a one-to-one correspondence between Cases 1--4 and the sets $I^{q}_{0}$, $I^{r}_{0}$, $I_{int}$ and~$I_{1}$.
  
  It remains to discuss the limiting case $\alpha_{i}^{2}=0$. Here the expected profit is independent of $K_{-i}$ and there is no incentive to produce with $r_{i}<1$. We readily see that either $A-c\leq Q_{-i}$ and $q_{i}=0$---that is, $i\in I^{q}_{0}$---or $A-c> Q_{-i}$ and $q_{i}=\frac12(A-c-Q_{-i})>0$; i.e., $i\in I_{1}$. With the conventions $\beta_{i}=0$, $a_{i}=\infty$ and $\beta_{i}a_{i}=d$, these are indeed the statements of the proposition in this case.
\end{proof}

The following will be helpful to prove the existence of an equilibrium.

\begin{remark}\label{rk:aPrioriBound}
  (a) The optimal quantity $q_{i}$ is continuous in $(Q_{-i},K_{-i})\in [0,A]\times\R_{+}$. Indeed, in each of the four cases of Proposition~\ref{pr:FOC}, $q_{i}$ is expressed as a continuous function $q_{i}=\varphi_{i}(Q_{-i},K_{-i})$. Each case is specified as a region in terms of $(Q_{-i},K_{-i})$ and the union of these regions is the whole space $[0,A]\times\R_{+}$. It remains to note that the functions $\varphi_{i}$ connect continuously at the boundaries. Similarly, $q_{i}$ can be represented as a continuous function of $(Q,K)$, and the same is true for $k_{i}$ instead of $q_{i}$.

  (b) The optimal quantity $q_{i}$ satisfies $q_{i}\leq (A-c)/2$.%
\end{remark} 

\begin{remark} \label{rk:exclusionPrinciple}
  In a given equilibrium, at most one of the following can occur:
  (i) some firm produces zero quantity (i.e., $I^{q}_{0}\neq\emptyset$), or
  (ii) some firm makes effort to reduce emissions (i.e., $I^{r}_{0}\cup I_{int}\neq \emptyset$).
  Indeed, Proposition~\ref{pr:FOC} shows that (i) implies $Q<z$ whereas (ii) implies $Q\geq z$. Of course, it is possible that neither (i) nor (ii) hold, meaning that all firms belong to $I_{1}$.
  To understand this exclusion economically, note that $Q\geq z$ implies $Q_{-i}\geq z=A-c-d$ for any firm $i$, in which case firm $i$ would certainly not want to produce at a cost of $c+d$ per unit. More generally, opting for $r_{i}\in(0,1)$ would be equivalent to producing part of the quantity at price $c$ and the rest at price $c+d$, and the latter again cannot be optimal. Thus, only $r_{i}=0$ is possible.
  Case (ii) is clearly the more relevant for our model. We can note that in this case, the definitions simplify to 
  $
      I_{0}^{q}=\emptyset, \quad
      I_{0}^{r}=\{i:\, K\geq a_{i}\},\quad
      I_{int}=\{i:\, K< a_{i} \mbox{ and } Q-K<z-a_{i}\},\quad
      I_{1}=\{i:\, K< a_{i}\mbox{ and }Q-\geq z-a_{i} \}.
  $
\end{remark}

The remainder of this section establishes existence and uniqueness of the equilibrium. We start with the straightforward part.

\begin{proof}[Proof of Theorem~\ref{th:mainBody}---Existence of Equilibrium]
  Fix a firm $i$ and consider arbitrary choices $(q_{j},k_{j})_{j\neq i}$ for the other firms, where $(q_{j},k_{j})\in D:=[0,(A-c)/2]^{2}$, as well as external carbon $K_{ex}\geq0$. (While we can equivalently use $(q_{j},r_{j})$ or $(q_{j},k_{j})$ to characterize a strategy, we opt for the latter in this proof because $k_{j}$'s continuity properties are more obvious.) Set $Q_{-i}=\min\{\sum_{j\neq i} q_{j},A\}$ and $K_{-i}=K_{ex}+\sum_{j\neq i} k_{j}$. 
  As mentioned in Remark~\ref{rk:aPrioriBound}\,(a), there exists an optimal response $(q_{i},k_{i})$ which depends continuously on $(Q_{-i},K_{-i})$ and hence is also continuous if seen as a function of $(q_{j},k_{j})_{j\neq i}$:
  $$
    (q_{i},k_{i})=\Phi_{i}((q_{j},k_{j})_{j\neq i}).
  $$
  Moreover, $\Phi_{i}$ maps into $D$ by Remark~\ref{rk:aPrioriBound}\,(b). 
  Forming a
  vector $\Phi$ from the functions $\Phi_{1},\dots,\Phi_{n}$ yields a map from $D^{n}$ into itself with the following property: if $(q_{j},k_{j})_{1\leq j\leq n}$ is a fixed point of $\Phi$ such that $Q:=\sum_{j} q_{j}\leq A$, then $(q_{j},k_{j})_{1\leq j\leq n}$ is a Nash equilibrium. Indeed, the latter condition on $Q$ ensures that $Q_{-i}=\sum_{j\neq i} q_{j}$ and then $(q_{i},k_{i})$ is the optimal response to the other firm's choices $(q_{j},k_{j})_{j\neq i}$. Since $\Phi$ is continuous and $\emptyset\neq D^{n}\subseteq \R^{2n}$ is compact and convex, Brouwer's fixed point theorem \citep[Corollary~17.56, p.\,583]{AliprantisBorder.06} implies that $\Phi$ has at least one fixed point. 
  
  Let $(q_{j},k_{j})_{1\leq j\leq n}$ be any fixed point and suppose for contradiction that $Q\geq A$. As $A>0$, there is at least one firm $i$ with $q_{i}>0$. We see from Proposition~\ref{pr:FOC} that $Q_{-i}=A$ implies $q_{i}=0$, so we must have $Q_{-i}<A$. But all cases in Proposition~\ref{pr:FOC} yield that $q_{i}<A-Q_{-i}$ as soon as $Q_{-i}<A$, and hence $Q=q_{i}+Q_{-i}<A$. As a result, any fixed point of $\Phi$ satisfies  $\sum_{j} q_{j}< A$ and is a Nash equilibrium. This completes the proof of existence.
\end{proof}  

\subsection{Proof of Uniqueness}

Throughout this proof we consider one equilibrium denoted as above with $(q_{j},r_{j})_{1\leq j\leq n}$, $k_{j}=r_{j}q_{j}$, $Q=\sum_{j} q_{j}$, $K=K_{ex}+\sum_{j}k_{j}$, etc., and a second equilibrium for the same parameters whose quantities are denoted with prime (i.e., $q'_{j}$, $Q'$, \dots). Our aim is to show that the two equilibria coincide.

\begin{lemma}\label{le:antimonotone}
  If $K'\geq K$ and $Q'\geq Q$, the two equilibria coincide.
\end{lemma} 

\begin{proof}
  As discussed in Remark~\ref{rk:aPrioriBound}\,(a), the optimal quantity $q_{i}$ of firm $i$ is a continuous function $q_{i}=\varphi(K,Q)$, and Proposition~\ref{pr:FOC}  shows that $\varphi$ is nonincreasing in both~$Q$ and~$K$. 
In particular, if $K'\geq K$ and $Q'\geq Q$, it follows that 
  $
    q'_{i}=\varphi_{i}(K',Q')\leq \varphi_{i}(K,Q)=q_{i}.
  $
Summing this over $i$ yields $Q'\leq Q$ and we deduce that $Q'= Q$ and $q'_{i}=q_{i}$ for all $1\leq i\leq n$. The same arguments apply to $k'_{i}$ and $k_{i}$.
\end{proof}

\begin{lemma}\label{le:groupFeedback}
   Let $G\subseteq (I^{r}_{0}\cup I_{int})$ consist of $m\in\{0,\dots,n\}$ firms and let $H$ denote the remaining $n-m$ firms. If $Q_{G}=\sum_{j\in G} q_{j}$ and $Q_{H}=Q-Q_{G}$ are the total quantities of those groups in a given equilibrium, then
  $$
    Q_{G} = \frac{m}{1+m}(z-Q_{H}).
  $$    
  Consider a second equilibrium (denoted with primes) and assume $G\subseteq (I^{r\prime}_{0}\cup I'_{int})$. Then
  $$
    Q'_{G} - Q_{G} =  - \frac{m}{1+m}(Q'_{H}-Q_{H})\quad\mbox{and}\quad Q' - Q = \frac{1}{1+m}(Q'_{H}-Q_{H}).
  $$  
\end{lemma}

\begin{proof}
  Proposition~\ref{pr:FOC} shows that all $i\in G$ produce the common quantity 
  $
    q_{i}=z-Q= z- Q_{G} - Q_{H}.
  $
  Summing over $i\in G$ yields that $Q_{G} = m(z - Q_{G} - Q_{H})$ which is the first assertion.
  The same result holds in the second equilibrium and now the assertion about $Q'_{G}$ follows by taking differences. 
\end{proof} 

As $\frac{m}{1+m}\in [0,1)$, Lemma~\ref{le:groupFeedback} shows that given a change in $Q_{H}$, the group~$G$ would react by partially, but not fully, compensating that change. This property of strategic substitutes is the driving force in the following key lemma.

\begin{lemma}\label{le:downcrossingUniqueness}
  Let $K'\geq K$ and $(I^{q\prime}_{0}\cup I'_{1})\subseteq (I^{q}_{0}\cup I_{1})$. Then the two equilibria coincide.
\end{lemma} 

\begin{proof}
  Let $G=I^{r}_{0}\cup I_{int}$ and $H=I^{q}_{0}\cup I_{1}$. The assumption ensures that 
  $G\subseteq (I^{r\prime}_{0}\cup I'_{int})$.
  We claim that $k'_{i}\leq k_{i}$ for all $i\in G$. Indeed, this is trivial if $i\in I^{r\prime}_{0}$. If $i\in I_{int}\cap I'_{int}$, it follows immediately from $K'\geq K$ and the formulas for $k_{i},k'_{i}$ in Proposition~\ref{pr:FOC}\,(iii). Finally, $i\in I^{r}_{0}\cap I'_{int}$ is impossible since it would imply that  $K\geq a_{i}>K'$. Thus, the claim holds and in particular $K'_{G}\leq K_{G}$ (notation of Lemma~\ref{le:groupFeedback}). As a consequence, any increase in total carbon must come from $H$; that is, $K'_{H}\geq K_{H}$. 

Note that firms $i\in H=I^{q}_{0}\cup I_{1}$ satisfy $q_{i}=k_{i}$ (either both quantity and carbon are zero or $r_{i}=1$) and thus $K_{H}=Q_{H}$. On the other hand, $r'_{j}\leq 1$ for all firms, so that $Q'_{H}\geq K'_{H}$. Therefore, $K'_{H}\geq K_{H}$ yields that $Q'_{H}\geq Q_{H}$. Now Lemma~\ref{le:groupFeedback} shows that $Q'\geq Q$ and we conclude by applying Lemma~\ref{le:antimonotone}.
\end{proof} 

It is intuitive that if we (exogenously) add carbon to an equilibrium, any firm that previously used  green technology will continue to do so---and even to a larger extent. This suggests that the second condition in Lemma~\ref{le:downcrossingUniqueness} is always verified, as confirmed by the following.

\begin{lemma}\label{le:upcrossingUniqueness}
  Let $K'\geq K$. Then $(I^{q\prime}_{0}\cup I'_{1})\subseteq (I^{q}_{0}\cup I_{1})$.
\end{lemma} 

\begin{proof}
  We use the inequalities in Proposition~\ref{pr:FOC} to show that each of the possible violations leads to a contradiction.  %
  
  Let $i\in I'_{1}\cap I^{r}_{0}$, then $K\geq a_{i}$ and hence $K'\geq a_{i}$. As $i\in I'_{1}$ means in particular that $z-Q'+\beta_{i}(a_{i}-K')>0$, it follows that $z-Q'>0$. Together, we have $z-Q'>0\geq a_{i}-K'$, which contradicts $i\in I'_{1}$. 
  
  Next, suppose that $i\in I'_{1}\cap I_{int}$.  Then $i\in I'_{1}$ yields $Q-K<z-a_{i}$ whereas $i\in I'_{1}$ implies $Q'-K'\geq z-a_{i}$. Thus, $Q'-Q\geq K'-K\geq0$. Now Lemma~\ref{le:antimonotone} shows that the equilibria coincide and in particular $I'_{1}\cap I_{int}=\emptyset$.
  
  Let $i\in I^{q\prime}_{0}\cap I^{r}_{0}$, then $Q'\geq z$ and $z> Q$, thus $Q'\geq Q$ and we again obtain a contradiction via Lemma~\ref{le:antimonotone}.
  
  Finally, if $i\in I^{q\prime}_{0}\cap I_{int}$, then $Q-z<K-a<0\leq Q'-z$ and in particular $Q'\geq Q$, and we conclude by Lemma~\ref{le:antimonotone}.
\end{proof}

\begin{proof}[Proof of Theorem~\ref{th:mainBody}---Uniqueness of Equilibrium]
  Given two equilibria, we may label them such that $K'\geq K$. Then Lemma~\ref{le:upcrossingUniqueness} yields $(I^{q\prime}_{0}\cup I'_{1})\subseteq (I^{q}_{0}\cup I_{1})$ and now Lemma~\ref{le:downcrossingUniqueness} shows that the two equilibria coincide.
\end{proof} 

\section{Proofs and Elaboration for Section~\ref{se:examples}}\label{se:appendixExamples}

In Section~\ref{se:twoFirms} we discussed the case $n=2$ of two firms under the condition that $z>d$ or equivalently $A>c+2d$. Here, we treat the complete state space of possible parameters. Recall that a firm producing nothing is called \emph{white}---it clearly emits zero carbon, but this is not necessarily because the firm cares about climate impacts. As before we label the two firms such that $a_{1}\leq a_{2}$.

In the trivial case $A-c\leq 0$, consumers are not willing to pay the marginal cost of production and hence both firms are white.
The case where $A-c>0$ but $z\leq 0$ has two regimes: white-white if $a_1 > (1/d +2/a_2)^{-1}$ and white-red if $a_1 \leq (1/d +2/a_2)^{-1}$. All these are special cases of Example~\ref{ex:noGreenTech}.

We now turn to the nondegenerate case $z>0$. With respect to Section~\ref{se:twoFirms}, there is an additional regime ``white-red.'' Its appearance mandates that some of the other regimes carry additional parameter restrictions (which are always verified when $z>d$ as in Section~\ref{se:twoFirms}); this concerns the red-red and green-red regimes. The other regimes are unchanged.

\begin{figure}[htb]
\begin{center}
\begin{tikzpicture}[scale=.60]%
\fill [color = myorange, opacity=0.9] (0,0) -- (1,2) -- (3,3) -- (2,1) -- cycle;
\fill [pattern = flexcheckerboard_greenorange, opacity=0.9] (0,0) -- (1,2) -- (0,2) -- cycle;
\fill [pattern = flexcheckerboard_greenorange, opacity=0.9] (0,0) -- (2,1) -- (2,0) -- cycle;

\fill [pattern = flexcheckerboard_greenred, opacity=0.9] (2,0) -- (2,1) --
plot [smooth,domain=1:3] ({16*\x/(9-\x)},\x) -- (8,0) -- cycle;
\fill [pattern = flexcheckerboard_greenred, opacity=0.9] (0,2) -- (1,2) --
plot [smooth,domain=1:3] (\x,{16*\x/(9-\x)}) -- (0,8) -- cycle;

\fill [pattern = flexcheckerboard_orangered, opacity=0.9] (2,1) -- 
plot [smooth,domain=1:3] ({16*\x/(9-\x)},\x) -- (3,3) -- cycle;
\fill [pattern = flexcheckerboard_orangered, opacity=0.9] (1,2) -- 
plot [smooth,domain=1:3] (\x,{16*\x/(9-\x)}) -- (3,3) -- cycle;

\fill [pattern = flexcheckerboard_whitered, opacity=0.9] (8,0) -- (8,3) --
plot [smooth,domain=3:3.53] ({24*\x/(12-\x)},\x) -- (10,0) -- cycle;
\fill [pattern = flexcheckerboard_whitered, opacity=0.9] (0,8) -- (3,8) --
plot [smooth,domain=3:3.53] (\x, {24*\x/(12-\x)}) -- (0,10) -- cycle;

\fill [color=myred,opacity=0.9] (3,3) -- (8,3) -- 
plot [smooth,domain=3:3.53] ({24*\x/(12-\x)},\x) -- (10,10) -- (3.53,10) --
plot [smooth,domain=3.53:3] (\x, {24*\x/(12-\x)}) -- cycle;

\draw[->] (0,0) -- (10.5,0) node[right] {$a_1$};
\draw[->] (0,0) -- (0,10.5) node[above] {$a_2$};

\draw (0,0) -- (2,1) -- (3,3) {};
\draw (0,0) -- (1,2) -- (3,3) {};

\draw (2,0) -- (2,1) {};
\draw (0,2) -- (1,2) {};

\draw (3,3) -- (3,8) -- (0,8) {};
\draw (3,3) -- (8,3) -- (8,0) {};

\draw [smooth,domain=1:3] plot (\x,{16*\x/(9-\x)});
\draw [smooth,domain=1:3] plot ({16*\x/(9-\x)},\x);

\draw [smooth,domain=3:3.53] plot (\x,{24*\x/(12-\x)});
\draw [smooth,domain=3:3.53] plot ({24*\x/(12-\x)},\x);

\draw (3,0) -- (3,-0.1) node[below] {$z$};
\draw (0,3) -- (-0.1,3) node[left] {$z$};
\end{tikzpicture}
\vspace{-1em}
\end{center}
\caption{Regimes of equilibria for the case of two. The white-red regimes cease to exist when $z> d$.}
\label{fi:twoFirmsAppendix}
\end{figure}
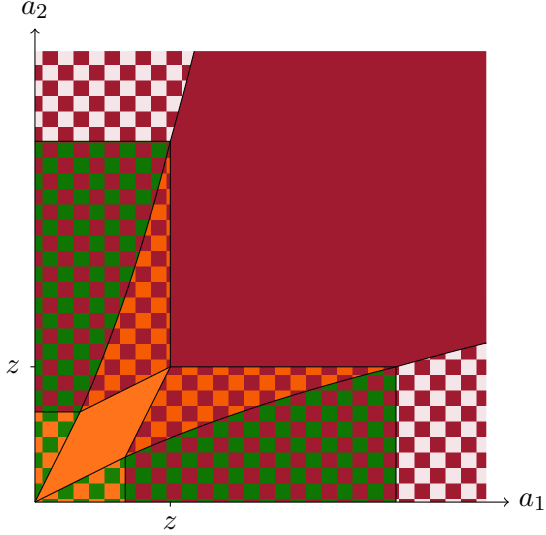

\begin{enumerate}
\item The regimes green-green, orange-orange, green-orange and orange-red are unchanged.

\item Red-red has the additional restriction $a_1a_2 > d(a_2-2a_1)$. %

\item Green-red has the additional restriction $(d-z)a_2 < 2zd$. %

\item \emph{White-red.} Suppose that $(d-z)a_2 \geq 2zd$ and $d(a_2-2a_1) \geq a_1a_2$. Then  $q_{1}=k_{1}=0$ and $q_2=k_2=\frac{A-c}{2(1+\beta_2)}$. This regime does not exist when $z>d$ as that renders the first condition impossible. Moreover, as seen in Figure~\ref{fi:twoFirmsAppendix}, this regime occurs only when $a_{1}$ is very large and $a_{2}$ is small, or vice versa. 
\end{enumerate}

Finally, we report the proofs that were omitted in Section~\ref{se:examples}.

\begin{proof}[Proof for Examples~\ref{ex:allInt} and~\ref{ex:noneAtOne}]
  It is straightforward that check that $Q=\sum_{j}q_{j}$ and $K=K_{ex}+\sum_{j}r_{j}q_{j}$  with the stated definitions. The optimality for each firm can then be proved by directly verifying the optimality conditions of Proposition~\ref{pr:FOC}. 
\end{proof}

\begin{proof}[Proof for Example~\ref{ex:noGreenTech}]
  Note that $a_{1}\leq \cdots \leq a_{n}$ is equivalent to  $\xi_{1}\leq \cdots\leq \xi_{n}$ and that $a_{1}\geq z + K_{ex}$ implies $a_{j}\geq z + K_{ex}$ for $1\leq j\leq n$.
  It is straightforward that check that $Q=\sum_{j}q_{j}$  with the stated definitions. Using the definition of $n_{0}$, we verify that $\frac{1}{1+\beta_{i}}[z-Q+\beta_{i}(a_{i}-K)]=\xi_{i}-Q\geq0$ for $i>n_{0}$ whereas the same quantity is $<0$ for $i\leq n_{0}$. To conclude that firms $i>n_{0}$ belong to $I_{1}$ (cf.\ Proposition~\ref{pr:FOC}), it only remains to observe  $z-a_{i}\leq z-a_{1} \leq -K_{ex} = Q-K$. Whereas for $i\leq n_{0}$ to belong to $I^{q}_{0}$, we need to establish that $Q\geq z$. Indeed, we can assume that $n_{0}\geq1$ as otherwise there is nothing to prove. Then the definition of $n_{0}$ leads to
  $
    Q = \frac{1}{n_{1}+1}\sum_{j > n_{0}} \xi_{j}> \xi_{n_{0}}\geq \xi_{1}.
  $
  Noting that the assumed inequality $a_{1}\geq z + K_{ex}$ can be rearranged into $\xi_{1}\geq z$, we conclude that $Q\geq z$ as desired. In summary, the optimality criteria of $i>n_{0}$ are the ones of $I_{1}$ and the criteria of $i\leq n_{0}$ are the ones of $I^{q}_{0}$. Since the stated formulas of $q_{i}$ and $k_{i}$ correspond to the ones of Proposition~\ref{pr:FOC} for these respective cases, $(q_{j},k_{j})_{1\leq j\leq n}$ indeed yield an equilibrium. Finally, the equilibrium is unique by Theorem~\ref{th:mainBody}.
\end{proof} 

\section{More on Comparative Statics}\label{se:compStatDetails}

We first derive several relationships between key quantities in the  equilibrium $(k_j,q_j)_{1 \leq j \leq n}$ associated
with a given set $(\alpha_j^2)_{1\leq j \leq n}$ of climate beliefs and fixed parameters $A,b,c,d > 0$ and 
$K_{ex} \geq 0$. Recall the definitions $I_0^r, I_{int}, I_1$ of Proposition~\ref{pr:FOC}. We will also find it useful to abbreviate 
$$
  G=I_0^r \cup I_{int}, \quad A_{int} = \sum_{j \in I_{int}}a_j,\quad B_1 = \sum_{j \in I_1}(1+\beta_j)^{-1},\quad Q_G = \sum_{j \in G} q_{j} 
$$
and similarly $Q_1 = \sum_{j \in I_1} q_{j}$, $K_{G} = \sum_{j \in G} k_{j}$ and $K_{1} = \sum_{j \in I_{1}} k_{j}$, as well as
$$
  n_0 = |I_0^r|,\quad n_{int} = |I_{int}|,\quad n_1= |I_1|,\quad m = |G| = n_0+n_{int}.
$$

\begin{proposition}\label{pr:KQexplicit}
In equilibrium, the following relations hold:
\begin{align}
Q & =Q_{0}+Q_{int}+Q_{1},\qquad K=K_{G}+K_{1}+K_{ex}, \qquad K_{1}=Q_{1}, \label{eq:lin1}\\
Q_G &= \frac{m}{m+1}(z-Q_1) = m(z-Q), \qquad Q = \frac{mz + Q_1}{m+1}, \label{eq:lin2}\\
K_G &= \frac{1}{n_{int}+1}A_{int} - \frac{n_{int}}{n_{int}+1}(K_{ex}+Q_1) = A_{int}-n_{int}K, \label{eq:lin3}\\
Q_1 &= B_1(A-c-Q_G+K_G+K_{ex})-n_1K, \label{eq:lin4} \\
K&= \frac{B_1(A-c+md) + (B_1+m+1)(A_{int}+K_{ex})}{(n_{int}+n_1+1)(m+1)-B_1n_0}, \label{eq:linK} \\
Q&= z + \frac{B_1(A-c+n_{int}d) - (n_{int}+n_1+1)z+(B_1-n_1)(A_{int}+K_{ex})}{(n_{int}+n_1+1)(m+1)-B_1n_0}. \label{eq:linQ} 
\end{align}
\end{proposition}

\begin{proof}
The relations in~\eqref{eq:lin1} are clear and~\eqref{eq:lin2} is a special case of Lemma~\ref{le:groupFeedback}. By Proposition~\ref{pr:FOC} we have $k_i = a_i - K$ for $i \in I_{int}$ and $k_{i}=0$ for $i\in I_0^r$, so that summing over $i\in G$ yields $K_{G}=K_{int}=A_{int}-n_{int}K$ and now $K-K_{G}=Q_{1}+K_{ex}$ yields~\eqref{eq:lin3}. For $i \in I_1$, Proposition~\ref{pr:FOC} yields that  
\begin{align*}
q_i &= (1+\beta_i)^{-1}(A-c-Q-\beta_iK) = (1+\beta_i)^{-1}(A-c-Q+K) - K \\
&= (1+\beta_i)^{-1}(A-c-Q_G + K_G + K_{ex}) - K
\end{align*}
and then summing over $i\in I_1$ shows~\eqref{eq:lin4}.
Equations~\eqref{eq:lin1}--\eqref{eq:lin4} yield the linearly independent system
\begin{align*}
  (B_{1}+m+1) Q + (n_{1}-B_{1}) K &= B_{1}(A-c)+mz,\\
  (m+1)Q-(n_{int}+1)K &= -K_{ex}-A_{int}+mz
\end{align*} 
for $Q,K$ which can then be solved to give~\eqref{eq:linK} and~\eqref{eq:linQ}.
\end{proof}
Next, we use the above formulas to compute the directional derivatives and, in particular, derive the results in Section~\ref{se:compStat}. In all that follows, we may and will assume that the direction of perturbation keeps the equilibrium in the same regime; that is, the sets $I^{r}_{0},I_{int},I_{1}$ and hence also the numbers $n_{0},n_{int},n_{1}, A_{int},B_{1}$ are constant during the perturbation. As explained in the beginning of Section~\ref{se:compStat}, this is always true after defining the boundaries appropriately (which corresponds to choosing appropriately the strict and non-strict inequalities in the definitions of $I^{r}_{0},I_{int},I_{1}$ in Proposition~\ref{pr:FOC}).
Denote by
$
  N=(n_{int}+n_1+1)(m+1)-B_1n_0>0
$
the denominator of~\eqref{eq:linK} and~\eqref{eq:linQ}. 
Using the formulas in Proposition~\ref{pr:KQexplicit} and setting $\tilde{n}=n_{0}+n_{int}+n_{1}$, we have
\begin{align*}
&\frac{\partial K}{\partial B_1} = \frac{(m + 1)[(A_{int}+K_{ex})(\tilde{n}+1) + (n_{int}+n_1 + 1)(z + (m+1)d)]}{N^2} > 0, \\
&\frac{\partial Q}{\partial B_1} = \frac{(n_{int} + 1)[(A_{int}+K_{ex})(\tilde{n}+1) + (n_{int}+n_1 + 1)(z + (m+1)d)]}{N^2} > 0, \\
&\frac{\partial K}{\partial A_{int}} = \frac{m + 1 + B_1}{N} > 0, \qquad \frac{\partial Q}{\partial A_{int}} = \frac{B_1-n_1}{N} \leq 0.
\end{align*}
(Of course, the derivatives with respect to $B_{1}$ make sense only when $I_{1}$ is not empty---otherwise there is no corresponding perturbation---and similarly for $A_{int}$.)
Combining these derivatives with the formulas for $q_{i}$ and $k_{i}$ in Proposition~\ref{pr:FOC} as well as~\eqref{eq:lin2} and~\eqref{eq:lin3}, we can then compute the following.

\begin{enumerate}
	\item Let $i \in I_0^q$. Then
	$$
	\frac{\partial k_i}{\partial A_{int}} = \frac{\partial k_i}{B_1} = 0,\qquad \frac{\partial q_i}{\partial A_{int}} = - \frac{\partial Q}{\partial A_{int}} \geq 0, \qquad 
	\frac{\partial q_i}{\partial B_1} = - \frac{\partial Q}{\partial B_1} < 0.
	$$

	\item Let $i \in I_{int}$. Then the formulas for $q_{i}$ in (i) still hold. In addition,
	\begin{align*}
	\frac{\partial k_i}{\partial B_1} &= - \frac{\partial K}{\partial B_1} < 0, %
	\qquad
	\frac{\partial k_i}{\partial a_j} = - \frac{\partial K}{\partial A_{int}} < 0 \text{ for }i \neq j\in I_{int}, \\
	\frac{\partial k_i}{\partial a_i} &= 1 - \frac{\partial K}{\partial A_{int}} \geq \frac{n_{int}(n+1)}{N} > 0, \\
  \frac{\partial r_i}{\partial a_j} &= \frac{1}{q_i}\Big(r_i\frac{\partial Q}{\partial A_{int}} - \frac{\partial K}{\partial A_{int}} \Big) < 0 \text{ for }i \neq j\in I_{int},  \\
  \frac{\partial r_i}{\partial a_i} &= \frac{1}{q_i}\Big(1 + r_i\frac{\partial Q}{\partial A_{int}} - \frac{\partial K}{\partial A_{int}} \Big) %
  = \frac{1}{q_i N}(N - m - 1 - B_1 + r_i(B_1 - n_1)) 
  > 0, \\
  \frac{\partial r_i}{\partial B_1} &= \frac{1}{q_i}\Big(r_i\frac{\partial Q}{\partial Q_1} - \frac{\partial K}{\partial Q_1}\Big)\frac{\partial Q_1}{\partial B_1} 
  = \frac{1}{q_i}\Big(\frac{r_i}{m+1} - \frac{1}{n_{int} + 1}\Big)\frac{\partial Q_1}{\partial B_1} < 0.
  \end{align*}

	\item Let $i \in I_1$ and set $b_j = (1+\beta_j)^{-1}$ for $j\in I_{1}$. Then for all $i\neq j\in I_{1}$,
	\begin{align}
	\frac{\partial k_i}{\partial A_{int}} &= \frac{\partial q_i}{\partial A_{int}}= \frac{(n_1 - B_1) - (1-b_i)(\tilde{n}+1)}{N} \in \R, \label{eq:ambiguousQuantity}\\
	\frac{\partial k_i}{\partial b_j} &= \frac{\partial q_i}{\partial b_j} = - \frac{(1-b_i)n_0 + n_{int} + 1}{n_0 + n_{int} + 1}\frac{\partial K}{\partial B_1} < 0, %
	\qquad
	\frac{\partial k_i}{\partial b_i} = \frac{\partial q_i}{\partial b_i}  > 0. \nonumber 
	\end{align}
\end{enumerate}

\begin{remark}\label{rk:ambiguousQuantity}
The sign of the derivative $\frac{\partial k_i}{\partial A_{int}}$ in~\eqref{eq:ambiguousQuantity} is ambiguous in specific circumstances. Suppose first that $i$ is the firm with the smallest coefficient $b_{i}$ among the firms in $I_{1}$, or equivalently, the largest $\alpha_{i}^{2}$. Then $B_{1}\geq n_{1}b_{i}$ and hence $\tilde{n}+1>n_{1}$ implies
\begin{equation}\label{eq:ambiguousQuantity}
  \frac{\partial k_i}{\partial A_{int}}= \frac{(n_1 - B_1) - (1-b_i)(\tilde{n}+1)}{N}<0
\end{equation}
which is the same sign as $\frac{\partial k_i}{\partial a_j}$ in (ii). 
Conversely, let $i$ be the firm with the largest coefficient $b_{i}$ among the firms in $I_{1}$ (hence among all firms). Then it may happen that 
the expression in~\eqref{eq:ambiguousQuantity} is strictly positive.
For instance, in the extreme case $\alpha_{i}^{2}=0$ we have $b_{i}=1$ and $\frac{\partial k_i}{\partial A_{int}}= \frac{n_1 - B_1}{N}$ is strictly positive as soon as there exists some firm $j\in I_{1}$ with $\alpha_{j}^{2}>0$.

The intuition for this phenomenon is as follows. Suppose that a firm in $I_{int}$ decreases climate belief, which corresponds to an increase in $A_{int}$. Then, neglecting equilibrium effects, it would emit more carbon and but keep the quantity constant.
(Other firms in $I_{int}$ may partially compensate this; the cumulative change of $I_{int}$ would still be an increase in carbon, with a reduction in quantity.) 
A firm $j\in I_{1}$ thus faces an environment with larger carbon emissions and similar production quantity. If $b_{j}$ is relatively small, the reaction is the expected one: firm $j$ emits less carbon and thus also produces less. In fact, this reduced production more than compensates the increase in quantity from $I_{int}$. Suppose now that there is some other firm $i\in I_{int}$ with a small coefficient $\alpha_{i}^{2}$; then this firm's optimality condition is hardly affected by the change in carbon; however, the decrease in quantity from the aforementioned firm $j$ has an outsized impact and leads firm $i$ to produce more, and hence also emit more carbon.
\end{remark}

\begin{remark}
The derivatives of $K$ and $Q$ with respect to $K_{ex}$ are the same as with respect to $A_{int}$; this follows from the fact that~\eqref{eq:linK} and~\eqref{eq:linQ} depend on the sum $A_{int}+K_{ex}$ rather than the individual quantities. In particular,
$Q$ is decreasing with respect to $K_{ex}$ whereas $K$ is increasing. The carbon emission $K - K_{ex}$ from the firms, on the other hand, is decreasing as
$
  \frac{\partial(K - K_{ex})}{\partial K_{ex}} = \frac{1}{N}(m + 1 + B_1 - N) < 0.
$
\end{remark}

\section{Proofs for Section~\ref{se:dynamics}}\label{se:appendixDynamics}

\begin{proof}[Proof of Proposition~\ref{pr:dynamics}]
  Let $K_{m}$ be the total carbon at the end of the $m$-th round. The sequence $(K_{m})$ is monotone increasing, hence convergent. In any given round~$m$, the optimality conditions of Proposition~\ref{pr:FOC} hold with $K=K_{m-1}+\sum_{j}k_{j}$ and $Q=\sum_{j}q_{j}$.
  
  Suppose for contradiction that $\lim K_{m}> a$. Let $m\geq1$ be the first round such that $K_{m}> a$, then in particular $K_{m}> a^{(m)}_{j}$ for all $1\leq j\leq n$. In view of Proposition~\ref{pr:FOC}, all firms are in $I^{q}_{0}\cup I^{q}_{0}$, hence no firm emits carbon and $K_{m}=K_{m-1}$. This contradicts the choice of $m$. (In fact the same argument shows that if $a>0$, then $K_{m}$ is strictly smaller than $a$ for all $m$; that is, the limit is not reached in finite time.)

  It remains to show that $\lim K_{m} \geq a$. Assume first that $Q_{m}\geq z$ holds for infinitely many rounds $m$. Proposition~\ref{pr:FOC} shows that all firms are in $I_{0}^{q}\cup I_{1}$, so that the additional carbon in round $m$ is $K_{m}-K_{m-1}=Q_{m}\geq z$. As $z>0$, the presence of infinitely many such rounds implies that $\lim K_{m} =\infty \geq a$ as desired. If the first assumption does not hold, there exists $m_{0}$ such that $Q_{m}< z$ for all $m\geq m_{0}$. Define $a^{(m)}=\max\{a^{(m)}_{1},\dots,a^{(m)}_{n}\}$. If $m\geq m_{0}$ and $a^{(m)}<\infty$, then
\begin{equation}\label{eq:geomGrowth}
  K_{m}\geq \max\big\{K_{m-1}, \lambda a^{(m)}+ (1-\lambda)K_{m-1})\big\}
\end{equation}
for some $\lambda>0$ independent of $m$. This follows by noting that $K_{m}\geq K_{m-1}$ and applying Lemma~\ref{le:minimalCarbon} below to the firm~$i$ with $a^{(m)}_{i}=a^{(m)}$. Whereas if $m\geq m_{0}$ and $a^{(m)}=\infty$, we check directly that $k_{i}\geq \frac12 (A-c-Q_{-i})\geq d/2$ and hence 
\begin{equation}\label{eq:geomGrowthExeption}
  K_{m}\geq K_{m-1} + d/2.
\end{equation}
The combination of~\eqref{eq:geomGrowth} and~\eqref{eq:geomGrowthExeption} shows that $\lim K_{m} \geq \limsup a^{(m)} =a$, and that completes the proof that $\lim K_{m}=a$.
\end{proof}

\begin{lemma}\label{le:minimalCarbon}
  Consider an equilibrium with $Q< z$. Then every firm~$i$ satisfies $k_{i}\geq \lambda (a_{i}-K_{-i})$ for $\lambda:=\frac12\frac{b\alpha_{i}^{2}}{1+b\alpha_{i}^{2}}>0$.
\end{lemma} 

\begin{proof}
  Let $i$ be any firm. As $Q< z$, we have $i\notin I^{q}_{0}$ in Proposition~\ref{pr:FOC}. If $i\in I^{q}_{0}$, then $a_{i}-K_{-i}<0$ and the claim is trivial. For $i\in I_{int}$ the claim follows from Proposition~\ref{pr:FOC}\,(iii), even with $\lambda=1/2$. For $i\in I_{int}$, note that $Q< z$ implies $Q_{-i}< z$ and then $A-c-Q_{-i}\geq d=b\alpha_{i}^{2}a_{i}$. Thus Proposition~\ref{pr:FOC}\,(iv) yields 
  \begin{align*}
    k_{i}
    =q_{i}
    =\frac12\frac{1}{1+b\alpha_{i}^{2}} [ A-c - Q_{-i} - b\alpha_{i}^{2} K_{-i}]
    \geq \frac12\frac{b\alpha_{i}^{2}}{1+b\alpha_{i}^{2}} [a_{i} - K_{-i}]
  \end{align*} 
  as claimed.
\end{proof}

\begin{proof}[Proof of Proposition~\ref{pr:dynamicsNoGreen}]
We may assume that $A>c$; otherwise no goods are produced and the result is clear. Note that $c+d\geq A$ (i.e., $z\leq0$) implies $z\leq Q$ and hence $I^{r}_{0}=I_{int}=\emptyset$ in any round; cf.\ Proposition~\ref{pr:FOC}.
 Fix a round $m$ and let $i$ be such that $\beta^{(m)}_{i}=\min\{\beta^{(m)}_{1},\dots,\beta^{(m)}_{n}\}$. Suppressing $m$ in the notation and assuming first that $\beta_{i}>0$, the formula for $Q$ in Example~\ref{ex:noGreenTech} shows that
 \begin{align*}
    Q&=\frac{1}{n_{1}+1}\Bigg( \sum_{j\in I_{1}} \frac{A-c -\beta_{j}K_{m-1}}{1+\beta_{j}}\Bigg) \geq \frac{1}{n_{1}+1}\left( \frac{A-c -\beta_{i}K_{m-1}}{1+\beta_{i}}\right) \\
    &\geq \frac{\beta_{i}}{(n_{1}+1)(1+\beta_{i})} \left(\frac{A-c}{\beta_{i}} -K_{m-1}\right)
    \geq \frac{\beta}{(n+1)(1+\beta)} \left(\frac{A-c}{\beta_{i}} -K_{m-1}\right)
\end{align*}
  and hence 
  \begin{equation}\label{eq:geomGrowthNoGreen1}
    K_{m}=K_{m-1}+Q_{m}\geq \lambda \frac{A-c}{\beta_{i}} + (1-\lambda)K_{m-1}
  \end{equation}
  for $\lambda=\frac{\beta}{(n+1)(1+\beta)}>0$. Whereas if $\beta_{i}=0$, the same line of argument yields 
    \begin{equation}\label{eq:geomGrowthNoGreen2}
    Q_{m}\geq \frac{A-c}{n+1}.
  \end{equation}
  At least one of \eqref{eq:geomGrowthNoGreen1} and \eqref{eq:geomGrowthNoGreen2} holds in any round $m$ and we conclude that $\lim K_{m}\geq (A-c)/\beta$ (i.e., $\lim K_{m}=\infty$ if $\beta=0$).
  
  It remains to show that $\lim K_{m}\leq (A-c)/\beta$ if $\beta>0$. Again, fix a round~$m$. As $1/(1+x)$ is decreasing and $x/(1+x)$ is increasing in $x$, Example~\ref{ex:noGreenTech} yields
 \begin{align*}
    Q&=\frac{1}{n_{1}+1}\Bigg( \sum_{j\in I_{1}} \frac{A-c -\beta_{j}K_{m-1}}{1+\beta_{j}}\Bigg)
    \leq \frac{n_{1}}{n_{1}+1}\left( \frac{A-c -\beta_{i}K_{m-1}}{1+\beta_{i}}\right) \\
    &\leq \frac{n_{1}}{n_{1}+1}\left( \frac{A-c -\beta_{i}K_{m-1}}{\beta_{i}}\right)
    < \frac{A-c}{\beta_{i}} -K_{m-1}.
 \end{align*} 
  Thus, $K_{m}=K_{m-1}+Q<(A-c)/\beta_{i}\leq (A-c)/\beta$ and the result follows.
\end{proof} 

\section{Other Utility Functions}\label{se:otherUtility}

In this appendix we discuss a generalization where consumers' utility function $u$ is not quadratic. For simplicity we focus on the interior type of equilibria with $K_{ex}=0$ and a smooth utility function $u:\R_{+}\to\R$. Stated in terms of the quantity $q_{i}$ and the carbon $k_{i}=r_{i}q_{i}$, firm $i$'s expected profit is
\begin{align*}%
  E_{i}[\pi_{i}(k_{i},q_{i})] 
  &= u'(Q)q_{i} - bE_{i}[\alpha^{2}]Kr_{i} q_{i} - [c+(1-r_{i})d] q_{i} \nonumber\\
  &=u'(q_{i}+Q_{-i})q_{i} - b\alpha_{i}^{2} (k_{i} + K_{-i})k_{i} - (c+d)q_{i}+dk_{i}.
\end{align*}
Interior maximizers satisfy $\partial_{q}E_{i}[\pi_{i}(k_{i},q_{i})] =0$ which with $Q_{-i}=\sum_{j\neq i}  q_{j}$ leads to the equation
$$
  u''(\tsum_{j} q_{j})q_{i} + u'(\tsum _{j}q_{j})= c+d, \quad 1\leq i\leq n.
$$
Given any solution $(q_{1},\dots q_{n})$, the equation implies that $q_{i}=\frac{c+d-u'(Q)}{u''(Q)}$ for $Q=\sum_{j} q_{j}$, for all $i$. That is, just as in Example~\ref{ex:allInt}, any solution consists of a common quantity $q_{0}$ for all firms, and clearly $q_{0}$ must be a solution of
\begin{equation}\label{eq:qFOCsymmetric}
  u''(\tsum_{j} q_{j})q_{i} + u'(\tsum_{j} q_{j})= c+d.
\end{equation}
As discussed below, this equation may have zero, one or more solutions for general $u$, but typical examples are well-behaved. Any solution induces an equilibrium as follows: suppose that~\eqref{eq:qFOCsymmetric} has an interior solution $q_{0}$ and that the associated optimal $k_{i}\in [0,r_{i}]$ are interior. Then the first-order condition $\partial_{k}E_{i}[\pi_{i}(k_{i},q_{i})] =0$ yields that $k_{i}=a_{i}-K$ for $K=\sum_{j} k_{j}$, or equivalently $r_{i}=k_{i}/q_{i}=(a_{i}-K)/q_{0}$, a direct extension of Example~\ref{ex:allInt}.

\begin{example}\label{ex:powerUtility}
  (a) For the logarithmic utility $u(x)=\log(x)$, the unique solution of~\eqref{eq:qFOCsymmetric} is 
  $
    q_{0} = \frac{n-1}{(c+d)n^{2}}
  $
  whenever $n\geq2$.
  
  (b) More generally, consider the CRRA utility $u(x)=(1-\gamma)^{-1}x^{1-\gamma}$ where $0<\gamma<n$ and $\gamma=1$ corresponds to the logarithmic case. Then %
  the unique solution of~\eqref{eq:qFOCsymmetric} is
  $
    q_{0} = \sqrt[\gamma]{ \frac{n-\gamma}{(c+d) \gamma n^{\gamma+1}} }.
  $
  For $\gamma\geq n$, no positive solution exists.
\end{example}

The following is a general sufficient condition for existence.

\begin{proposition}
  Define the relative risk aversion $\rho(x)=-\frac{x u''(x)}{u'(x)}$. Suppose that $u$ satisfies the Inada conditions $u'(0)=\infty$ and $u'(\infty)=0$ and that $\sup_{0\leq x \leq \eps}\rho(x)<n$ for some $\eps>0$. Then~\eqref{eq:qFOCsymmetric} has a solution. If moreover $nu'''(nx)x + (n+1)u''(nx) <0$,  the solution is unique.
\end{proposition}

\begin{proof}
  Note that $\rho(x)$ is nonnegative and that
  $$
    \varphi(x):=u''(nx)x + u'(nx) = u'(nx)\bigg[1+ \frac{u''(nx)x}{u'(nx)}\bigg]=u'(nx)\bigg[1- \frac{1}{n}\rho(nx)\bigg]
  $$
  is a continuous function of $x$. Under the stated conditions, $\varphi(0)=\infty$ and $\varphi(\infty)\leq0$. Thus, existence of a solution follows from the intermediate value theorem. Moreover, uniqueness must hold when $\varphi$ is strictly decreasing, and that is implied by the fact that 
  $
    \varphi'(x)=nu'''(nx)x + (n+1)u''(nx) <0.
  $
\end{proof}

\end{document}